\documentclass[12pt,a4paper]{article}

\usepackage[utf8]{inputenc}
\usepackage[english]{babel}
\usepackage{lmodern}
\usepackage[usenames, dvipsnames]{xcolor}
\usepackage{amssymb}
\usepackage{amsfonts}
\usepackage{amsmath}
\usepackage{amsthm}
\usepackage{mathtools}
\usepackage{dsfont}

\newtheorem{theorem}{Theorem}[section]
\newtheorem{lemma}[theorem]{Lemma}
\newtheorem{remark}[theorem]{Remark}
\newtheorem{proposition}[theorem]{Proposition}
\newtheorem{corollary}[theorem]{Corollary}
\theoremstyle{definition}
\newtheorem{example}[theorem]{Example}
\newtheorem{definition}[theorem]{Definition}
\newtheorem{question}[theorem]{Open Question}

\allowdisplaybreaks[4]

\newcommand{\bN}{\mathbb{N}}
\newcommand{\bQ}{\mathbb{Q}}
\newcommand{\bZ}{\mathbb{Z}}
\newcommand{\bR}{\mathbb{R}}
\newcommand{\bC}{\mathbb{C}}

\newcommand{\bE}{\mathbb{E}}

\newcommand{\cB}{\mathcal{B}}

\newcommand{\cL}{\mathcal{L}}

\newcommand{\cP}{\mathcal{P}}

\newcommand{\1}{\mathds{1}}
\newcommand{\ii}{\textnormal{i}}

\newcommand{\re}{\mathrm{e}}
\newcommand{\ri}{\mathrm{i}}
\newcommand{\di}{\mathrm{d}}
\newcommand{\law}{\mathcal{L}}
\newcommand{\QID}{\textnormal{QID}}

\DeclareMathOperator{\supp}{supp}
\DeclareMathOperator{\trace}{trace}
\DeclareMathOperator{\qid}{\textnormal{q.i.d.}}
\DeclareMathOperator{\dx}{d\textit{x} \hspace{1pt}}

\DeclareMathOperator{\dz}{d\textit{z} \hspace{1pt}}

\begin{document}

\title{On multivariate quasi-infinitely divisible distributions\thanks{In honour of Ron Doney on the occasion of his 80th birthday. This research was supported by DFG grant LI-1026/6-1. Financial support is gratefully acknowledged.}}

\author{David Berger\thanks{Technische Universit\"at Dresden,
Institut f\"ur Mathematische Stochastik, D-01062 Dresden, Germany, email: david.berger2@tu-dresden.de} \and Merve Kutlu \and Alexander Lindner\thanks{Ulm University, Institute of Mathematical Finance, D-89081 Ulm, Germany;
emails: merve.kutlu@uni-ulm.de,
alexander.lindner@uni-ulm.de}}
\date{\today}
\maketitle

\begin{abstract}
A quasi-infinitely divisible distribution on $\bR^d$ is a probability distribution $\mu$ on $\bR^d$ whose characteristic function can be written as the quotient of the characteristic functions of two infinitely divisible distributions on $\bR^d$. Equivalently, it can be characterised as a probability distribution whose characteristic function has a L\'evy--Khintchine type representation with a \lq\lq signed L\'evy measure\rq\rq, a so called quasi--L\'evy measure, rather than a L\'evy measure. A systematic study of such distributions in the univariate case has been carried out in Lindner, Pan and Sato \cite{lindner}. The goal of the present paper is to collect some known results on multivariate quasi-infinitely divisible distributions and to extend some of the univariate results to the multivariate setting. In particular, conditions for weak convergence, moment and support properties are considered. A special emphasis is put on examples of such distributions and in particular on $\bZ^d$-valued quasi-infinitely divisible distributions.
\end{abstract}



\section{Introduction}
The class of L\'evy processes may be considered to be one of the most important classes of stochastic processes. On the one hand, L\'evy processes generalise Brownian motion, stable L\'evy processes or compound Poisson processes, on the other hand they are the natural continuous time analogue  of random walks and as those are interesting both from a purely theoretical point of view, but also from a practical point of view as drivers of stochastic differential equations, similar to time series being driven by i.i.d. noise. Excellent books on L\'evy processes include the expositions by Applebaum \cite{Applebaum2009}, Bertoin \cite{Bertoin1996}, Doney~\cite{Doney2007}, Kyprianou \cite{Kyprianou2014} or Sato \cite{sato}. We do not attempt to summarise Ron Doney's numerous, important and deep contributions to the understanding of L\'evy processes, but confine ourselves to mentioning some of his works on small and large time behaviour of L\'evy processes, such as Bertoin et al. \cite{BertoinDoneyMaller2008}, Doney \cite{Doney2004} or Doney and Maller \cite{DoneyMaller2002}, which have been of particular importance for the research of one of the authors of this article.

The class of $\bR^d$-valued L\'evy processes corresponds naturally to the class of infinitely divisible distributions on $\bR^d$, which is   an important and  well-studied class of distributions, see e.g. \cite{sato} for various of its properties. Infinitely divisible distributions (and hence L\'evy processes) are completely characterised by the L\'evy-Khintchine formula, according to which $\mu$ is infinitely divisible if and only if its
characteristic function $\bR^d \ni z \mapsto \widehat{\mu}(z) = \int_{\bR^d} \re^{\ri \langle z, x\rangle} \, \mu(\di x)$  can be expressed for all $z\in \bR^d$ as
\begin{align} \label{eq-LK}
\widehat{\mu}(z)
= \exp\left(
	\ii \langle \gamma, z \rangle
	- \frac{1}{2} \langle z, Az \rangle
	+ \int_{\bR^d} \left(
		e^{\ii \langle z,x \rangle} - 1
		- \ii \langle z,x \rangle \1_{[0,1]}(|x|)
	\right)\nu(\dx)
\right),
\end{align}
with a symmetric non-negative definite matrix $A \in \bR^{d \times d}$, a constant $\gamma \in \bR^d$ and a L\'evy measure $\nu$ on $\bR^d$, that is, a Borel measure on $\bR^d$ satisfying
$\nu(\{0\}) = 0$ and
$\int_{\bR} (1 \land |x|^2) \nu(\dx) < \infty$.
The triplet $(A,\nu,\gamma)$ is unique and called the \emph{(standard) characteristic triplet} of the infinitely divisible distribution $\mu$. By \eqref{eq-LK}, the characteristic function of an infinitely divisible distribution must obviously be zero-free.

The class of quasi-infinitely divisible distributions on $\bR^d$ is much less known. Their definition is as follows (see Lindner et al. \cite[Remark 2.4]{lindner}):

\begin{definition} \label{def-qid}
A probability distribution $\mu$ on $\bR^d$ is called \emph{quasi-infinitely divisible}, if its characteristic function $\widehat{\mu}$ admits the representation $\widehat{\mu}(z) = \widehat{\mu}_1(z) / \widehat{\mu}_2(z)$ for all $z\in \bR^d$ with infinitely divisible distributions $\mu_1$ and $\mu_2$.
\end{definition}
Rewriting this as $\widehat{\mu}(z) \widehat{\mu}_2(z) = \widehat{\mu}_1(z)$, we see that a probability distribution $\mu$ on $\bR^d$ is quasi-infinitely divisible if and only if there are two infinitely divisible distributions $\mu_1$ and $\mu_2$ on $\bR^d$ such that
$$\mu \ast \mu_2 = \mu_1.$$ So quasi-infinitely divisible distributions arise naturally in the factorisation problem of infinitely divisible distributions, where one of the factors is infinitely divisible, and the other is then necessarily quasi-infinitely divisible. If $\mu, \mu_1,\mu_2$ are related as in Definition~\ref{def-qid}, and if $(A_1,\nu_1,\gamma_1)$ and $(A_2,\nu_2,\gamma_2)$ denote the characteristic triplets of $\mu_1$ and $\mu_2$, respectively, then it is easy to see that the characteristic function $\widehat{\mu}$ of $\mu$ has a L\'evy--Khintchine type representation as in \eqref{eq-LK}, with $\gamma = \gamma_1-\gamma_2$, $A = A_1 - A_2$ and $\nu = \nu_1 - \nu_2$, where in the definition of $\nu$ one has to be a bit careful since $\nu_1(B) - \nu_2(B)$ will not be defined for  Borel sets $B\subset \bR^d$ with $\nu_1(B) = \nu_2(B) = \infty$. It is however defined if $\min \{\nu_1(B), \nu_2(B) \} < \infty$, in particular if $B$ is bounded away from zero. We will formalise this in Definition \ref{qlm} and call $\nu=\nu_1-\nu_2$ a \emph{quasi-L\'evy (type) measure}, so basically a \lq\lq signed\rq\rq~L\'evy measure with some extra care taken for sets that are not bounded away from zero. In Theorem \ref{lk} we shall then give a L\'evy--Khintchine type formula for quasi-infinitely divisible distributions with these quasi-L\'evy (type) measures.

Quasi-infinitely divisible distributions have already appeared (although not under this name) in early works of Linnik \cite{Linnik64}, Linnik and Ostrovski\u{i} \cite{LO77}, Gnedenko and Kolmogorov \cite{GK68} or Cuppens \cite{Cuppens1969,cuppens75}, to name just a few, but a systematic study of them in one dimension was only initiated in \cite{lindner}. The name \lq\lq quasi-infinitely divisible\rq\rq~ for such distributions seems to have been used the first time in Lindner and Sato \cite{LiSa2011}.

The class of quasi-infinitely divisible distributions is larger than it might appear on first sight. For example, Cuppens \cite[Prop. 1]{Cuppens1969}, \cite[Thm.~4.3.7]{cuppens75} showed that a probability distribution on $\bR^d$ that has an atom of mass greater than $1/2$ is quasi-infinitely divisible, and  in \cite[Theorem 8.1]{lindner} it was shown that a distribution supported in $\bZ$ is quasi-infinitely divisible if and only if its characteristic function has no zeroes, which in \cite[Theorem 3.2]{cramerwold} was extended to $\bZ^d$-valued distributions. For example, a binomial distribution $b(n,p)$ on $\bZ$ is quasi-infinitely divisible if and only if $p\neq 1/2$. Quasi-infinite divisibility of one-dimensional distributions of the form $\mu = p \delta_{x_0} + (1-p) \mu_{ac}$ with $p\in (0,1]$ and an absolutely continuous $\mu_{ac}$ on $\bR$ has been characterised in Berger \cite[Thm. 4.6]{berger}.
Further, as shown in \cite[Theorem 4.1]{lindner}, in dimension 1 the class of quasi-infinitely divisible distributions on $\bR$ is dense in the class of all probability distributions with respect to weak convergence. Since there are probability distributions that are not quasi-infinitely divisible, the class of quasi-infinitely divisible distributions obviously cannot be closed.

Recently, applications of quasi-infinitely divisible distributions have been found in physics (Demni and Mouayn \cite{DemniMouayn2015}, Chhaiba et al. \cite{ChhaibaDemniMouayn2016}) and insurance mathematics (Zhang et al. \cite{ZhangLiuLi}). Quasi-infinitely divisible processes and quasi-infinitely divisible random measures and integration theory with respect to them have been considered in Passeggieri \cite{Passeggeri2020a}. Quasi-infinitely divisible distributions have also found applications in number theory, see e.g. Nakamura \cite{Nakamura13, Nakamura15} or Aoyama and Nakamura \cite{AN13}.
We also mention the recent work of Khartov \cite{Khartov2019}, where compactness criteria for quasi-infinitely divisible distributions on $\bZ$ have been derived, and the paper by Kadankova et al. \cite{KadankovaSimonWang2020}, where an example of a quasi-infinitely divisible distribution on the real line whose quasi-L\'evy measure has a strictly negative density on $(a_*,0)$ for some $a_* < 0$ has been constructed.

As mentioned, a systematic study of one-dimensional quasi-infinitely divisible distributions has only been initiated in \cite{lindner}.
The goal of the present paper is to give a systematic account of quasi-infinitely divisible distributions on $\bR^d$. We will collect some known results and also extend some of the one-dimensional results in \cite{berger} and \cite{lindner} to the multivariate setting. To get a flavour of the methods, we have decided to include occasionally also proofs of already known results, such as the previously mentioned result of Cuppens \cite{cuppens75}, according to which a probability distribution on $\bR^d$ with an atom of mass greater than $1/2$ is quasi-infinitely divisible.

The paper is structured as follows. In Section \ref{S2} we will formalise the concept of quasi-L\'evy measures and the L\'evy--Khintchine type representation of quasi-infinitely divisible distributions. We show in particular that the matrix $A$ appearing in the characteristic triplet must be non-negative definite (Lemma \ref{nnd}) and that the introduction of complex quasi-L\'evy type measures and complex symmetric matrices does not lead to new distributions, in the sense that if a probability distribution has a L\'evy--Khintchine type representation with a complex symmetric matrix $A\in \bC^{d\times d}$ and a complex valued quasi-L\'evy measure, then $A\in \bR^{d\times d}$ and $\nu$ is real valued, i.e. a quasi-L\'evy measure (Theorem \ref{clk}). In Section \ref{S3} we give examples of quasi-infinitely divisible distributions, in particular we reprove Cuppens' result (Theorem \ref{cuppens}) and state some of the examples mentioned previously. We also show how to construct multivariate quasi-infinitely divisible distributions from independent one-dimensional quasi-infinitely divisible distributions. Section~\ref{S4} is concerned with sufficient conditions for absolute continuity of quasi-infinitely divisible distributions, by extending a classical condition of Kallenberg \cite[pp.794-795]{kallenberg} for one-dimensional infinitely divisible distributions to multivariate quasi-infinitely divisible distributions (Theorem \ref{abscont}). This condition seems to be new even in the case of multivariate infinitely divisible distributions. Section \ref{S5} is concerned with topological properties of the class of quasi-infinitely divisible distributions, like it being dense with respect to weak convergence in dimension 1. In Section \ref{S6} we give a sufficient condition for weak convergence of quasi-infinitely divisible distributions in terms of their characteristic triplets, and in Section \ref{S7} we consider some support properties. Section \ref{S8} is concerned with moment conditions for quasi-infinitely divisible distributions and formulae for the moments in terms of the characteristic triplet. We end this section by setting some notation.

Throughout, we denote by $\bN=\{1,2,3,\ldots\}$ the natural numbers, by $\bN_0 = \bN \cup \{0\}$ the natural numbers including 0, and by $\bZ, \bQ, \bR$ and $\bC$ the integers, rational numbers, real numbers and complex numbers, respectively. The real part $a$ of a complex number $z=a+b\ri$ with $a,b\in \bR$ is denoted by $a = \Re(z)$, the imaginary part by $b=\Im(z)$, and the complex conjugate by $\overline{z} = a - b \ri$. Vectors in $\bR^d$ will be column vectors, we denote by $\bR^{n\times d}$ the set of all  $n\times d$ matrices with real entries, and the transpose of a vector or matrix $A$ is denoted by $A^T$.
The Euclidian inner product in the $d$-dimensional space $\bR^d$ is denoted by $\langle \cdot, \cdot \rangle$, and the absolute value of $z=(z_1,\ldots, z_d)^T$ by $|z|=(z_1^2 + \ldots + z_d^2)^{1/2}$. For two real numbers $a,b$ we denote the minimum of $a$ and $b$ by $a\wedge b$, and for a set $A$ the indicator function $\omega \mapsto \1_A(\omega)$ takes the value 1 for $\omega \in A$ and 0 for $\omega \notin A$. For two sets $B_1,B_2 \subset \bR^d$ and a vector $b\in \bR^d$ we write  $b + B_1 = \{ b + c: c \in B_1\}$ and $B_1 + B_2 := \{c+d : c\in B_1, d\in B_2\}$.
By a (probability) distribution on $\bR^d$ we will always mean a probability measure on $(\bR^d, \cB(\bR^d))$, where $\cB(\bR^d)$ denotes the Borel-$\sigma$-algebra on $\bR^d$. The Dirac measure at a point $x\in \bR^d$ is denoted by $\delta_x$, the convolution of two distributions $\mu_1$ and $\mu_2$ by $\mu_1\ast \mu_2$, and the product measure of them by $\mu_1\otimes \mu_2$, with $\mu_1^{\ast n}$ and $\mu_1^{\otimes n}$ denoting the $n$-fold convolution and $n$-fold product measure of $\mu_1$ with itself. The characteristic function $\bR^d \ni z \mapsto \int_{\bR^d} \re^{\ri \langle z, x \rangle} \, \mu(\di x)$ of a probability distribution $\mu$ on $\bR^d$ is denoted by $\widehat{\mu}$, the support of a non-negative measure $\nu$ on $\bR^d$  by $\supp (\nu)$, and weak convergence of probability measures is denoted by  $\stackrel{w}{\to}$. The law (or distribution) of an $\bR^d$-valued random vector will be denoted by $\law(X)$, its expectation by $\bE(X)$. By a signed measure on a $\sigma$-algebra we mean a $\sigma$-additive $[-\infty,\infty]$-valued set function that assigns the value 0 to the empty set, and we say that it is finite if it is $\bR$-valued. The restriction of a (signed) measure $\nu$ on a measurable space $(\Omega,\mathcal{F})$ to $\mathcal{A}\subset \mathcal{F}$ is denoted by $\nu_{|\mathcal{A}}$, and if $\mathcal{A}$ is of the form $\mathcal{A} = \{ F \cap A : F \in \mathcal{F}\}$ with some $A\in \mathcal{F}$ we occasionally also write $\nu_{|A}$ rather than $\nu_{\mathcal{A}}$. The support of a signed measure is the support of its total variation measure.

\section{The L\'evy-Khintchine type representation} \label{S2}
\setcounter{equation}{0}

As already mentioned, quasi-infinitely divisible distributions admit a L\'evy-Khintchine representation, with a quasi-L\'evy type measure instead of a L\'evy measure.
A quasi-L\'evy type measure is, in a sense, the difference between two L\'evy measures $\nu_1$ and $\nu_2$ and can be seen as a \lq\lq signed L\'evy measure\rq\rq.
However, this difference is not a signed measure if both, $\nu_1$ and $\nu_2$ are infinite.
On the other hand, for any neighborhood $U$ of $0$, the restrictions of $\nu_1$ and $\nu_2$ to $\bR^d \setminus U$ are finite, so that the difference is a finite signed measure.
The following concept of \cite{lindner} formalizes this statement.

\begin{definition} \label{qlm}
For $r > 0$ let $\cB_r^d \coloneqq \{B \in \cB(\bR^d): B \subset \{x\in \bR^d : |x|\geq r\}\}$ and let $\cB_0^d \coloneqq \cup_{r>0} \cB_r^d$.
Let $\nu: \cB_0^d \to \bR$ be a function such that $\nu_{|\cB_r^d}$ is a finite signed measure for every $r > 0$ and denote the total variation, the positive and the negative part of $\nu_{|\cB_r^d}$ by $|\nu_{|\cB_r^d}|$, $\nu^+_{|\cB_r^d}$ and $\nu^-_{|\cB_r^d}$, respectively.
\begin{enumerate}
\item[(a)] The \emph{total variation} $|\nu|$, the \emph{positive part} $\nu^+$ and the \emph{negative part} $\nu^-$ of $\nu$ are defined as the unique measures on $(\bR^d, \cB(\bR^d))$ that satisfy
\begin{align*}
|\nu|(\{0\})
= \nu^+(\{0\})
= \nu^-(\{0\})
= 0
\end{align*}
as well as
\begin{align*}
|\nu|(B)= |\nu_{|\cB_r^d}|(B), \
\nu^+(B)= (\nu_{|\cB_r^d})^+(B), \ \textrm{ and } \
\nu^-(B)= (\nu_{|\cB_r^d})^-(B)
\end{align*}
for all $B \in \cB_r^d$ and $r>0$.

\item[(b)] A function $f: \bR^d \to \bC$ is said to be \emph{integrable} with respect to $\nu$, if $f$ is integrable with respect to $|\nu|$, i.e. if $f$ is integrable with respect to $\nu^+$ and $\nu^-$.
In this case, we define
\begin{align*}
\int_{\bR^d} f(x) \nu(\dx)
\coloneqq \int_{\bR^d} f(x) \nu^+(\dx) - \int_{\bR^d} f(x) \nu^-(\dx).
\end{align*}

\item[(c)] $\nu$ is called a \emph{quasi-L\'evy type measure} on $\bR^d$, if the mapping $\bR^d \to \bR$, $x \mapsto 1 \land |x|^2$ is integrable with respect to $\nu$.
\end{enumerate}
\end{definition}

It is easy to see that the measures $|\nu|, \nu^+$ and $\nu^-$ defined in Definition \ref{qlm} are well-defined, where the uniqueness is observed by the condition that the point $0$ gets assigned no mass.
Note that the mapping $\nu$ itself is defined on $\cB_0^d$, which is not a $\sigma$-algebra, and hence $\nu$ is no signed measure.
But whenever $\nu$ has an extension on $\cB(\bR^d)$ which is a signed measure, we will identify $\nu$ with this extension and speak of $\nu$ as a signed measure.
Moreover, for two L\'evy measures $\nu_1$ and $\nu_2$ on $\bR^d$, the mapping $\nu \coloneqq (\nu_1)_{|\cB_0^d} - (\nu_2)_{|\cB_0^d}$ is obviously a quasi-L\'evy type measure on $\bR^d$.

Next, we give the L\'evy-Khintchine type representation for quasi-infinitely divisible distributions, which we immediately state for general representation functions: by a \emph{representation function} on $\bR^d$ we mean a bounded, Borel measurable function $c: \bR^d \to \bR^d$ that satisfies $\lim_{x \to 0}|x|^{-2}|c(x)-x| = 0$. The representation function $\bR^d \ni x\mapsto  x \1_{[0,1]}(|x|)$ is called the \emph{standard representation function}. It is well known that for every fixed representation function $c$, a probability distribution $\mu$ is infinitely divisible if and only if its characteristic function has a representation as in \eqref{lkf} below with $A\in \bR^{d\times d}$ being non-negative definite, $\gamma\in \bR^d$ and $\nu$ being a L\'evy measure on $\bR^d$; Equation \eqref{eq-LK} then corresponds to the use of the standard representation function. The triplet $(A,\nu,\gamma)$ is then unique and called the characteristic triplet of $\mu$ with respect to the representation function $c$, also denoted by $(A,\nu,\gamma)_c$, cf. \cite[Sect. 56]{sato}. Observe that only the location parameter $\gamma$ depends on the specific choice of the representation function. For the standard representation function we get the standard characteristic triplet. Let us now come to the L\'evy--Khintchine type representation of quasi-infinitely divisible distributions. This has already been observed in \cite[Rem. 2.4]{lindner}, but we have decided to give the proof in detail.

\begin{theorem} \label{lk}
Let $c$ be a representation function on $\bR^d$.
A probability distribution $\mu$ on $\bR^d$ is quasi-infinitely divisible if and only if its characteristic function $\widehat{\mu}$ admits the representation
\begin{align} \label{lkf}
\widehat{\mu}(z)
= \exp\left(
	\ii \langle \gamma, z \rangle
	- \frac{1}{2} \langle z, Az \rangle
	+ \int_{\bR^d} \left(
		e^{\ii \langle z,x \rangle} - 1
		- \ii \langle z, c(x) \rangle
	\right)\nu(\dx)
\right)
\end{align}
for all $z \in \bR^d$ with a symmetric matrix $A \in \bR^{d \times d}$, a constant $\gamma \in \bR^d$ and a quasi-L\'evy type measure $\nu$ on $\bR^d$.
In this case, the triplet $(A,\nu,\gamma)$ in the representation \eqref{lkf} of $\widehat{\mu}$ is unique.
\end{theorem}

\begin{proof}
It is clear that if $\mu$ is quasi-infinitely divisible with $\widehat{\mu}  = \widehat{\mu}_1 / \widehat{\mu}_2$ and $\mu_1,\mu_2$ being infinitely divisible with characteristic triplets $(A_i,\nu_i,\gamma_i)_c$, $i=1,2$, then $\widehat{\mu}$ has  the representation \eqref{lkf} with $A=A_1-A_2$, $\nu= \nu_1 - \nu_2$ and $\gamma=\gamma_1-\gamma_2$. Conversely, let $\mu$ be a probability distribution whose characteristic function $\widehat{\mu}$ admits the representation
 \eqref{lkf} with a symmetric matrix $A \in \bR^{d \times d}$, a constant $\gamma \in \bR^d$ and a quasi-L\'evy measure $\nu$ on $\bR^d$. Since $A\in \bR^{d\times d}$ is symmetric, we can write $A=A^+ - A^-$ with non-negative symmetric matrices $A^+,A^-\in \bR^{d\times d}$, which can be seen by diagonalising $A$, splitting the obtained diagonal matrix into a difference of two diagonal matrices with non-negative entries, and then transforming these matrices back.
Let $\mu_1$ and $\mu_2$ be infinitely divisible distributions with characteristic triplets $(A^+, \nu^+, \gamma)$ and $(A^-, \nu^-, 0)$, respectively.
Then $\mu_2 \ast \mu = \mu_1$, so that $\mu$ is quasi-infinitely divisible.

The uniqueness of the triplet is proved in Cuppens \cite[Thm. 4.3.3]{cuppens75} or also Sato \cite[Exercise 12.2]{sato}, but for clarity in the exposition we repeat the argument.
So let $(A_1,\nu_1,\gamma_1)$ and $(A_2,\nu_2,\gamma_2)$ be two triplets satisfying \eqref{lkf}. Defining
\begin{align*}
\Psi_j: \bR^d \to \bR,
	\quad z \mapsto \ii \langle \gamma_j, z \rangle
	-\frac{1}{2} \langle z,A_jz \rangle
	+\int_{\bR^d} \left(
		e^{\ii \langle z,x \rangle } -1
		-\ii \langle z,c(x) \rangle
	\right) \nu_j(\dx)
\end{align*}
for $j\in \{1,2\}$, it is easily seen that both $\Psi_1$ and $\Psi_2$ are continuous with $\Psi_j(0) = 0$, implying $\Psi_1 = \Psi_2$ by the uniqueness of the distinguished logarithm, cf. \cite[Lem. 7.6]{sato}.
As before we can find symmetric non-negative definite matrices $A_1^+, A_1^-, A_2^+$ and $A_2^-$ such that $A_1 = A_1^+-A_1^-$ and $A_2 = A_2^+-A_2^-$.
Therefore, the equation $\Psi_1 = \Psi_2$ can be rewritten to
\begin{align*}
	&\ii \langle \gamma_1, z \rangle
	-\frac{1}{2} \langle z,(A_1^+ + A_2^-)z \rangle
	+\int_{\bR^d} \left(
		e^{\ii \langle z,x \rangle } -1
		-\ii \langle z,c(x) \rangle
	\right) (\nu_1^+ +\nu_2^-)(\dx) \\
	& \quad =\ii \langle \gamma_2, z \rangle
	-\frac{1}{2} \langle z,(A_2^+ + A_1^-)z \rangle
	+\int_{\bR^d} \left(
		e^{\ii \langle z,x \rangle } -1
		-\ii \langle z,c(x) \rangle
	\right) (\nu_2^+ + \nu_1^-)(\dx)
\end{align*}
for all $z \in \bR^d$.
By the uniqueness of the L\'evy-Khintchine representation of infinitely divisible distributions (e.g. \cite[Thm. 8.1]{sato}), it follows that $\gamma_1 = \gamma_2$, $A_1^+ + A_2^- = A_2^+ +A_1^-$ and $\nu_1^+ +\nu_2^- = \nu_2^+ + \nu_1^-$, which implies that $A_1 = A_2$ and $\nu_1 = \nu_2$ (observe that $\nu_1^+,\nu_1^-,\nu_2^+$ and $\nu_2^-$ are all finite on $\cB_0^d$).
\end{proof}

\begin{definition} \label{triplet}
Let $c$ be a representation function on $\bR^d$. For a quasi-infinitely divisible distribution $\mu$ on $\bR^d$, the representation of $\widehat{\mu}$ in \eqref{lkf} is called the \emph{L\'evy-Khintchine representation} of $\mu$ and the function $\Psi_\mu : \bR^d \to \bC$ given by
\begin{align*}
	\Psi_\mu (z)
	\coloneqq \ii \langle \gamma, z \rangle
	-\frac{1}{2} \langle z,Az \rangle
	+\int_{\bR^d} \left(
		e^{\ii \langle z,x \rangle } -1
		-\ii \langle z,c(x) \rangle
	\right) \nu(\dx)
\end{align*}
for all $z \in \bR^d$ is called the \emph{characteristic exponent} of $\mu$.
The triplet $(A, \nu, \gamma)$ is called the \emph{generating triplet} or \emph{characteristic triplet} of $\mu$ \emph{with respect to} $c$ and denoted by $(A, \nu, \gamma)_c$.
The matrix $A$ is called the \emph{Gaussian covariance matrix} of $\mu$, the mapping $\nu$ the \emph{quasi-L\'evy measure} of $\mu$ and the constant $\gamma\in \bR^d$ the \emph{location parameter} of $\mu$ with respect to $c$. When $d=1$ we also speak of $A$ as the \emph{Gaussian variance} of $\mu$.
We write $\mu \sim \qid(A, \nu, \gamma)_c$ to state that $\mu$ is a quasi-infinitely divisible distribution with characteristic triplet $(A, \nu, \gamma)$ with respect to $c$. If $c$ is the standard representation function, then $(A, \nu, \gamma)_c$ is also called the \emph{(standard) characteristic triplet} of $\mu$ and is denoted by $(A, \nu, \gamma)$, omitting the index $c$.
\end{definition}

\begin{remark} \label{rem-triplet}
{\rm
(a) It is easily seen that the Gaussian covariance matrix and the quasi-L\'evy measure of a quasi-infinitely divisible distribution do not depend on the specific representation function, but the location parameter does.\\
(b) It is well known that the right-hand side of \eqref{lkf} defines the characteristic triplet of a probability distribution $\mu$ for all $\gamma\in \bR^d$, all L\'evy measures $\nu$ on $\bR^d$ and all non-negative definite symmetric $A\in \bR^{d\times d}$, in which case $\mu$ is necessarily infinitely divisible. It is however not true that the right-hand side of \eqref{lkf} defines the characteristic function of a probability distribution for all $\gamma\in \bR^d$, symmetric matrices $A\in \bR^{d\times d}$ and quasi-L\'evy type measures $\nu$. To see this, let $(A,\nu,\gamma)_c$ be the characteristic triplet of a quasi-infinitely divisible distribution such that $A$ is not non-negative definite or such that $\nu$ is not non-negative. If all such triplets were to give rise to  characteristic functions of a probability distribution, then in particular $(n^{-1} A, n^{-1} , n^{-1} \gamma)_c$ must be the characteristic triplet of some probability distribution $\mu_n$, say, for all $n\in \bN$. It is then easy to see that the characteristic function of the $n$-fold convolution of $\mu_n$ with itself has L\'evy-Khintchine type representation \eqref{lkf}, so that $\mu_n^{\ast n} = \mu$ for each $n\in \bN$. Hence $\mu$ is infinitely divisible, and the uniqueness of the characteristic triplet implies that $A$ is non-negative definite and $\nu$ is non-negative, a contradiction.  In Lemma \ref{nnd} below we will actually see that only non-negative definite matrices $A$ are possible. The quasi-L\'evy measure does not need to be a L\'evy measure, examples of which will be given in Section~\ref{S3}. However, not every quasi-L\'evy type measure can occur as the quasi-L\'evy measure of a quasi-infinitely divisible distribution, e.g. a quasi-L\'evy type measure $\nu$ in $\bR$ with $\nu^- \neq 0$ and $\nu^+$ being the zero measure or a one-point measure can never be the quasi-L\'evy measure of a quasi-infinitely divisible distribution, as mentioned in \cite[Ex. 2.9]{lindner}. This is the reason why we distinguish between quasi-L\'evy type measures and quasi-L\'evy measures. A quasi-L\'evy type measure is any function $\nu:\cB_0^d \to \bR$ as in Definition \ref{qlm}, while a quasi-L\'evy measure is a quasi-L\'evy type measure that is linked to a (necessarily quasi-infinitely divisible) probability distribution.
}
\end{remark}

In \cite[Lem. 2.7]{lindner} it was shown that if $(a,\nu,\gamma)$ is the characteristic triplet of a quasi-infinitely divisible distribution on $\bR$, then necessarily $a\geq 0$. We now extend this to higher dimensions, by showing that the Gaussian covariance matrix must necessarily be non-negative definite.

\begin{lemma} \label{nnd}
If $\mu$ is a quasi-infinitely divisible distribution on $\bR^d$ with characteristic triplet $(A, \nu, \gamma)$, then $A$ is non-negative definite.
\end{lemma}
\begin{proof}
For $z \in \bR^d$ and $t \in \bR$ it holds
\begin{align*}
	\Psi_{\mu}(tz)
	= \ii t \langle \gamma, z \rangle
	- t^2 \frac{1}{2} \langle z,Az \rangle
	+\int_{\bR^d} \left(
		e^{\ii t \langle z,x \rangle } -1
		-\ii t \langle z,x \rangle \1_{ [0,1] }(|x|)
	\right) \nu(\dx).
\end{align*}
Due to \cite[Lem. 43.11 (i)]{sato} we have
\begin{align*}
	\lim_{t \to \infty} t^{-2}\int_{\bR^d} \left(
		 e^{\ii t \langle z,x \rangle } -1
		 -\ii t \langle z,x \rangle \1_{ [0,1] }(|x|)
	\right) \nu^{\pm}(\dx)
	= 0,
\end{align*}
hence $ \lim_{t \to \infty} t^{-2}\Psi_{\mu}(tz)
= - \frac{1}{2} \langle z, Az \rangle$.
If $\langle z_0, Az_0 \rangle < 0$ would hold for some $z_0 \in \bR^d$, then we would obtain $|\widehat{\mu}(tz_0)| = |\exp(\Psi_{\mu}(tz_0))| \to \infty$ as $t \to \infty$, which is a contradiction.
\end{proof}

It is natural to ask why one should restrict to symmetric matrices $A\in \bR^{d\times d}$ in the L\'evy-Khintchine type representation and not allow arbitrary matrices $A\in \bR^{d\times d}$. The next remark clarifies that this does not lead to new distributions, but that one would loose uniqueness of the characteristic triplet when allowing more generally non-symmetric matrices.

\begin{remark}
{\rm Given an arbitrary matrix $A \in \bR^{d \times d}$ we can write $A = A_1 + A_2$, where $A_1 = \frac{1}{2}(A+A^{T})$ and $A_2 = \frac{1}{2}(A-A^{T})$.
The matrix $A_1$ is symmetric and $A_2$ satisfies $A_2^T = -A_2$, which implies that
$\langle z, A_2z \rangle
    = z^TA_2z
    = (z^T A_2z)^T
    = z^TA_2^Tz
    = -\langle z, A_2z \rangle$,
and therefore $\langle z, A_2z \rangle = 0$ for all $z \in \bR^d$.
It follows that $\langle z, Az \rangle = \langle z, A_1z \rangle$ for all $z \in \bR^d$.
Hence, if we do not require that the matrix $A$ in Theorem \ref{lk} is symmetric, then the representation of $\widehat{\mu}$ in \eqref{lkf} is not unique.
Further, the class of distributions $\mu$ on $\bR^d$ whose characteristic function $\widehat{\mu}$ allows the representation \eqref{lkf} with an arbitrary matrix $A \in \bR^{d \times d}$ is exactly the class of quasi-infinitely divisible distributions.}
\end{remark}

Having seen the reason why we restrict to symmetric matrices, we would now like to know if we get new distributions (or non-unique triplets) if we also allow for complex $\gamma\in \bR^d$,  complex symmetric $A\in \bC^{d\times d}$ and complex quasi-L\'evy measures in the L\'evy--Khintchine representation.
Berger showed in \cite[Thm. 3.2]{berger} that for $d=1$ this does not lead to a greater class of distributions, and that then necessarily $\gamma\in \bR$, $A\in [0,\infty)$ and that $\nu$ is real-valued, i.e. a quasi-L\'evy measure.
We now generalise this result to distributions on $\bR^d$.
To state this theorem, a \emph{complex quasi-L\'evy type measure} on $\bR^d$ is a mapping $\nu: \cB_0^d \to \bC$ such that $\Re \nu$ and $\Im \nu$ are quasi-L\'evy type measures on $\bR^d$.
A function $f: \bR^d \to \bC$ is said to be \emph{integrable} with respect to $\nu$, if it is integrable with respect to $\Re \nu$ and $\Im \nu$.
In this case, we define
\begin{align*}
    \int_{\bR^d} f(x) \nu(\dx)
    \coloneqq \int_{\bR^d} f(x) (\Re\nu)(\dx)
    + \ii \int_{\bR^d} f(x) (\Im\nu)(\dx).
\end{align*}

\begin{theorem} \label{clk}
Let $\mu$ be a distribution on $\bR^d$ such that
its characteristic function admits the representation
\begin{align*}
	\widehat{\mu}(z)
	=\exp \left(
		\ii \langle \gamma, z \rangle
		-\frac{1}{2} \langle z,Az \rangle
		+\int_{\bR^d} \left(
			e^{\ii \langle z,x \rangle } -1
			-\ii \langle z,x \rangle \1_{ [0,1] }(|x|)
		\right) \nu(\dx)
	\right)
\end{align*}
for every $z \in \bR^d$ with a symmetric matrix $A \in \bC^{d \times d}$, a complex quasi-L\'evy type measure $\nu$ on $\bR^d$ and $\gamma \in \bC^d$.
Then $A \in \bR^{d \times d}$, $\gamma \in \bR^d$ and $\Im \nu = 0$, that is, $\nu$ is a quasi-L\'evy type measure and $\mu$ is quasi-infinitely divisible.
\end{theorem}

\begin{proof}
The proof is very much the same as that of \cite[Thm. 3.2]{berger} in dimension 1, but we give the full proof for convenience.

For $z \in \bR^d$ we have
$$
    |\widehat{\mu}(z)|^2
    =\widehat{\mu}(z) \widehat{\mu} (-z)
        = \exp \left(
        -\langle z,Az \rangle + 2 \int_{\bR^d} \left(
            \cos \langle z,x \rangle -1
        \right) \nu(\dx)
    \right).$$
The function
$g: \bR^d \to \bC$, $
    z \mapsto -\langle z,Az \rangle + 2 \int_{\bR^d} \left(
        \cos \langle z,x \rangle -1
    \right) \nu(\dx)$
is continuous and satisfies $g(0) = 0$, implying that $g$ is the distinguished logarithm of $|\widehat{\mu}|^2$, see \cite[Lem. 7.6]{sato}.
The uniqueness of the distinguished logarithm implies that $g$ also has to be the natural logarithm of $|\widehat{\mu}|^2$, so that $g(z) \in \bR$ for all $z \in \bR^d$.
Hence,
\begin{align} \label{complex}
    - \frac{1}{2}\langle z,( \Im A)z \rangle + \int_{\bR^d} \left(
        \cos \langle z,x \rangle -1
    \right) (\Im\nu) (\dx)
    = 0
    \quad \textrm{for all } z \in \bR^d.
\end{align}
Further, for $z \in \bR^d$ it holds
$$
    \frac{\widehat{\mu}(z)}{\widehat{\mu}(-z)}
    = \exp \left(
        2 \ii \left(
            \langle \gamma, z \rangle
            + \int_{\bR^d} \left(
                \sin \langle z,x \rangle - \langle z,x \rangle \1_{[0,1]}(|x|)
            \right) \nu(\dx)
        \right)
    \right)$$
and $|\widehat{\mu}(z)| = |\overline{\widehat{\mu}(z)}| = |\widehat{\mu}(-z)|$, so $\left| \frac{\widehat{\mu}(z)}{\widehat{\mu}(-z)} \right| = 1$ and thus
\begin{align*}
    \langle \Im \gamma, z \rangle
    + \int_{\bR^d} \left(
        \sin \langle z,x \rangle - \langle z,x \rangle \1_{[0,1]}(|x|)
    \right) (\Im \nu)(\dx)
    = 0
    \quad \textrm{for all } z \in \bR^d.
\end{align*}
Adding this identity multiplied by $\ri$ to \eqref{complex} we obtain
\begin{align*}
   \Psi_{\delta_0}(z) =  0
    &= \ii \langle \Im \gamma ,z \rangle
    - \frac{1}{2} \langle z, (\Im A) z \rangle
    + \int_{\bR^d} \left(
        e^{\ii \langle z,x \rangle} - 1 - \ii \langle z,x \rangle \1_{[0,1]}(|x|)
    \right) (\Im \nu)(\dx)
\end{align*}
for all $z \in \bR^d$.
By the uniqueness of the L\'evy-Khintchine type representation for the Dirac measure $\delta_0$, it follows that $\Im A = 0$, $\Im \nu = 0$ and $\Im \gamma = 0$.
\end{proof}

In the proof of Lemma \ref{nnd}, through $\Psi_\mu(tz)$  we implicitly were concerned with the projections of quasi-infinitely divisible distributions onto the lines $\{tz : t \in \bR\}$ for given $z\in \bR^d$.
These projections are again quasi-infinitely divisible, and this holds more generally for affine linear images of random vectors with quasi-infinitely divisible distribution:

\begin{lemma} \label{projection}
Let $X$ be a random vector in $\bR^d$ with $\mu=\law(X)$ being  quasi-infinitely divisible with characteristic triplet $(A, \nu, \gamma)$. Let $b \in \bR^m$ and $M \in \bR^{m \times d}$.
Then the distribution of the $\bR^m$-valued random vector $U := M X + b$ is quasi-infinitely divisible
with characteristic triplet $(A_U, \nu_U, \gamma_U)$, where
\begin{align*}
    A_U &= MAM^{T},\\
    \gamma_U &= b + M \gamma
    + \int_{\bR^d} Mx \left(
        \1_{[0,1]}(|Mx|) - \1_{ [0,1] }(|x|)
    \right) \nu(\dx) \quad \textrm{and} \\
    \nu_U(B) &= \nu( \{ x \in \bR^d : Mx \in B\})
    \quad \textrm{for } B \in \cB_0^m.
\end{align*}
\end{lemma}
\begin{proof}
We see that
\begin{align*}
\widehat{\law(U)}(z)
= \int_{\bR^d} \re^{\ii \langle Mx +b, z \rangle} \, \mu(\dx)
= \re^{\ri \langle b , z \rangle} \int_{\bR^d} \re^{\ii \langle x, M^T z \rangle} \, \mu(\dx)
= \re^{\ri \langle b , z \rangle} \widehat{\mu}(M^T z)
\end{align*}
for $z\in \bR^m$.
The rest follows similar to \cite[Prop. 11.10]{sato}. 
\end{proof}

We conclude this section with a remark that it is also possible to define the drift or center of a quasi-infinitely divisible distribution, provided the quasi-L\'evy measure satisfies a certain integrability condition.

\begin{remark} \label{mean}
{\rm Let $\mu \sim \qid(A, \nu, \gamma)$ and suppose that
$
\int_{|x| \leq 1} |x|\,  |\nu|(\dx) < \infty.
$
Then the characteristic function $\widehat{\mu}$ of $\mu$ can be rewritten to
\begin{align} \label{lkdrift}
\widehat{\mu}(z)
= \exp \left(
	 \ii \langle \gamma_0, z \rangle
	- \frac{1}{2} \langle z, Az \rangle
	+ \int_{\bR^d} \left(
		e^{\ii \langle z,x \rangle} -1
	\right)\nu(\dx)
\right)
\end{align}
for all $z \in \bR^d$, where $\gamma_0 = \gamma - \int_{|x| \leq 1} x \nu(\dx)$.
This representation is unique and $\gamma_0$ is called the \emph{drift} of $\mu$.
Conversely, is $\mu$ is a distribution on $\bR^d$ such that its characteristic function admits the representation
\eqref{lkdrift} for a symmetric matrix $A \in \bR^{d \times d}$, a quasi-L\'evy measure $\nu$ on $\bR^d$ and $\gamma_0 \in \bR^d$,
then one can easily verify that $\mu$ is quasi-infinitely divisible with characteristic triplet $(A,\nu,\gamma)$, where $\gamma = \gamma_0 + \int_{|x| \leq 1} x \nu(\dx)$.
Then $(A, \nu, \gamma_0)$ is also called the \emph{characteristic triplet} of $\mu$ with respect to $c(x) = 0$ and denoted by $(A, \nu, \gamma_0)_0$.\\
Similarly, if
$
\int_{|x|>1} |x| \, |\nu|(\dx) < \infty,
$
then $\widehat{\mu}$ admits the representation
\begin{align} \label{lkmean}
\widehat{\mu}(z)
= \exp \left(
	 \ii \langle \gamma_m, z \rangle
	- \frac{1}{2} \langle z, Az \rangle
	+ \int_{\bR^d} \left(
		e^{\ii \langle z,x \rangle} -1 - \ii \langle z,x \rangle
	\right)\nu(\dx)
\right)
\end{align}
for all $z \in \bR^d$ with $\gamma_m = \gamma + \int_{|x| > 1} x \, \nu(\dx)$.
This representation is again unique and $\gamma_m$ is called the \emph{center} of $\mu$.}
\end{remark}

\section{Examples} \label{S3}
\setcounter{equation}{0}

A helpful tool to find examples of quasi-infinitely divisible distributions is the fact that the convolution of quasi-infinitely divisible distributions is again quasi-infinitely divisible.

\begin{remark} \label{convolution}
{\rm Let $c: \bR^d \to \bR^d$ be a representation function.
If $\mu_1 \sim \qid(A_1, \nu_1, \gamma_1)_c$ and $\mu_2 \sim \qid(A_2, \nu_2, \gamma_2)_c$, then
$\mu_1 \ast \mu_2
\sim \qid(A_1+A_2, \nu_1+\nu_2, \gamma_1+\gamma_2)_c.$}
\end{remark}

An important class of quasi-infinitely divisible distributions was established by Cuppens \cite[Prop. 1]{Cuppens1969}, \cite[Thm. 4.3.7]{cuppens75}. 
He showed that every distribution which has an atom of mass $\lambda > \frac{1}{2}$ is quasi-infinitely divisible. We state his result and also prove it, in order to get an idea of what is behind the theorem.

\begin{theorem} \label{cuppens}
Let $\mu = \lambda \delta_a + (1-\lambda) \sigma$ some $\lambda \in (\frac{1}{2},1]$, $a \in \bR^d$ and a distribution $\sigma$ on $\bR^d$ that satisfies $\sigma(\{a\}) = 0$.
Then $\mu$ is quasi-infinitely divisible with finite quasi-L\'evy measure
$\nu
= \left(\sum_{k = 1}^\infty
\frac{(-1)^{k+1}}{k} \left(
	\frac{1-\lambda}{\lambda}
\right)^k (\delta_{-a} \ast \sigma)^{\ast k}\right)_{|\cB_0^d}$,
Gaussian covariance matrix $0$ and drift $a$.
\end{theorem}

\begin{proof}
Shifting $\mu$ by $a$, we can and do assume without loss of generality by Remark \ref{convolution} that
$a = 0$. Define $\rho := \left(\sum_{k = 1}^\infty
\frac{(-1)^{k+1}}{k} \left(
	\frac{1-\lambda}{\lambda}
\right)^k \sigma^{\ast k}\right)$, where the sum converges absolutely to the finite signed measure $\rho$ since  $0 \leq \frac{1-\lambda}{\lambda} < 1$. Denote $\nu := \rho_{|\cB_0^d}$ as in the statement of the theorem, which then is a finite quasi-L\'evy type measure (observe that $\rho(\{0\})  \neq 0$ is possible although $\sigma(\{0\}) = 0$; hence it is important to subtract any mass of $\rho$ at 0, which is in particular achieved by restricting $\rho$ to $\cB_0^d$).

Next, observe that
$\widehat{\mu}(z)= \lambda + (1-\lambda) \widehat{\sigma}(z) = \lambda (1 + \frac{1-\lambda}{\lambda} \widehat{\sigma}(z))$ for  $z \in \bR^d$.
Again,
since $0 \leq \frac{1-\lambda}{\lambda} < 1$ and $|\widehat{\sigma}(z)|\leq 1$ for all $z \in \bR^d$, the series expansion of the principal branch of the complex logarithm of $\log (1+w)$ for $w\in \bC$ such that $|w|<1$ gives
\begin{align*}
    \log \widehat{\mu}(z)
    &= \log \lambda + \sum_{k=1}^{\infty}\frac{(-1)^{k+1}}{k} \left(
         \frac{1-\lambda}{\lambda}
    \right)^k \widehat{\sigma}(z)^k \\
    &= \log \lambda + \sum_{k=1}^{\infty}\frac{(-1)^{k+1}}{k} \left(
         \frac{1-\lambda}{\lambda}
    \right)^k \int_{\bR^d} \re^{\ii \langle z, x \rangle} \, \sigma^{\ast k}(\dx) \\
    &= \log \lambda + \int_{\bR^d} \re^{\ii \langle z, x \rangle} \, \rho(\dx)\\
    & =  \log \lambda + \rho(\bR^d) + \int_{\bR^d}  (\re^{\ii \langle z, x \rangle}-1) \, \nu(\dx)
\end{align*}
for every $z \in \bR^d$.
Using the fact that $\widehat{\mu}(0) = 1$, we find $0
    = \log \widehat{\mu}(0)
    = \log \lambda + \rho(\bR^d)$, finishing the proof.
\end{proof}

Lindner et al. \cite[Thm. 8.1]{lindner} showed that a probability distribution $\mu$ on $\bZ$ is quasi-infinitely divisible if and only if its characteristic function has no zeroes. This has been extended recently to distributions on $\bZ^d$ by Berger and Lindner \cite[Thm. 3.2]{cramerwold}. The precise result is as follows:

\begin{theorem} \label{gitter}
Let $\mu$ be a distribution that is supported in $\bZ^d$.
Then $\mu$ is quasi-infinitely divisible if and only if $\widehat{\mu}(z) \neq 0$ for all $z \in [0, 2\pi]^d$. In that case, the Gaussian covariance matrix of $\mu$ is zero, the quasi-L\'evy measure is finite and supported in $\bZ^d \setminus \{0\}$ and the drift of $\mu$ is in $\bZ^d$.
\end{theorem}

The proof given in \cite{cramerwold, lindner} relies on the L\'evy--Wiener theorem, according to which for a continuous function $f:\bR^d \to \bC$ that is $2\pi$-periodic in all coordinates and is such that it has  absolutely summable Fourier coefficients, and a holomorphic function $h:D \to \bC$ on an open subset $D\subset \bC$ such that $f([0,2\pi]^d) \subset \bC$, also the composition $h\circ f$ has absolutely summable Fourier coefficients (i.e. it is an element of the so called Wiener algebra). The given zero-free characteristic function $\widehat{\mu}$ then has to be modified appropriately in order to apply this Wiener--L\'evy theorem to the distinguished logarithm, and then an argument is needed in order to show that the Fourier coefficients are indeed real-valued and not complex. This is carried out in detail in \cite{cramerwold, lindner} and we refer to these articles for the detailed proof. We only mention here that it is also possible to replace the proofs given there for the fact that the Fourier coefficients are real and not complex by Theorem \ref{clk} given in the present article.

There is nothing special about the lattice $\bZ^d$ and Theorem \ref{gitter} continues to hold for more general lattices, which is the contents of the next result that generalises \cite[Cor. 3.10]{lindner} to higher dimensions.

\begin{corollary} \label{c-lattice}
Let $M\in \bR^{d\times d}$ be invertible, $b\in \bR^d$ and $\mu$ be a probability distribution supported in the lattice $M \bZ^d + b = \{M z + b: z \in \bZ^d\}$. Then $\mu$ is quasi-infinitely divisible if and only if the characteristic function of $\mu$ has no zeroes on $(M^T)^{-1} ([0,2\pi]^d) = \{ (M^T)^{-1} x : x\in [0,2\pi]^d\}$. In that case, the Gaussian covariance matrix of $\mu$ is zero, the quasi-L\'evy measure is finite and supported in $M \bZ^d\setminus \{0\}$ and the drift of $\mu$ is in $M \bZ^d +b$.
\end{corollary}

\begin{proof}
Let $U$ be a random vector with distribution $\mu$ and define $X=M^{-1} (U-b)$. Then $\law(X)$ is supported in $\bZ^d$ and $\widehat{\mu}(z) = \widehat{{\law(U)}}(z)= \re^{\ri \langle b, z \rangle} \widehat{\law(X)} (M^T z)$ for $z\in \bR^d$. The result is then an immediate consequence of Theorem \ref{gitter} together with Lemma \ref{projection}; here, an easy extension of Lemma \ref{projection} shows that the drift $\gamma_U$ of $U$  is $M \gamma_X + b$, where $\gamma_X$ is the drift of $\law(X)$.
\end{proof}

An interesting application of the previous theorem on $\bZ^d$ has been given in \cite[Thm. 4.1]{cramerwold}, where a Cram\'er--Wold device for infinite divisibility of $\bZ^d$-distributions was established. The precise statement is as follows:

\begin{corollary} \label{cor-cramerwold}
Let $X$ be a $\bZ^d$-valued random vector with distribution $\mu$. Then the following are equivalent:\\
(i) $\mu$ is infinitely divisible.\\
(ii) $\law(a^T X)$ is infinitely divisible for all $a\in \bR^d$.\\
(iii) $\law(a^T X)$ is infinitely divisible for all $a\in \bN_0^d$.\\
(iv) The characteristic function of $\mu$ has no zeroes on $\bR^d$ and there exists some $a=(a_1,\ldots, a_d)^T \in \bR^d$ such that $a_1,\ldots, a_d$ are linearly independent over $\bQ$ and such that $\law(a^T X)$ is infinitely divisible.
\end{corollary}

This result is striking in the sense that a Cram\'er-Wold device does not hold in full generality for infinite divisibility of $\bR^d$-valued distributions. Indeed, it is even known that for every $\alpha \in (0,1)$ there exists a  $d$-dimensional random vector $X$ such that $\law(a^T X)$ is $\alpha$-stable for all $a\in \bR^d$, but that $\law(X)$ is not infinitely divisible (see \cite[Sect. 2.2]{SamorodnitskyTaqqu}). The proof of Corollary \ref{cor-cramerwold} heavily relies on Theorem \ref{gitter}, and we refer to \cite[Thm. 4.1]{cramerwold} for the details of the proof.

In view of Theorem \ref{gitter} it is natural to ask if every distribution $\mu$ whose characteristic function is zero-free must be quasi-infinitely divisible. That this is not the case was shown in \cite[Ex. 3.3]{lindner} by giving a counter example. Let us give another counter example to this fact:

\begin{example} \label{ex-counter}
Consider the function $\varphi:\bR \to (0,\infty)$ given by $\varphi(x) = \exp (1- \re^{|x|})$. Then $\varphi$ is continuous on $\bR$, $\varphi(0) = 1$ and $\varphi''(x) = (\re^x - 1) \exp (1+x-\re^x)> 0$ for all $x>0$. Hence $\varphi$ is strictly convex on $(0,\infty)$. Since $\varphi(x)$ tends to 0 as $|x|\to \infty$, P\'olya's theorem implies that $\varphi$ is the characteristic function of an absolutely continuous distribution $\mu$ on $\bR$. The distinguished logarithm of $\mu$ is given by $\Psi_\mu(x) = 1 - \re^{|x|}$. Hence $\lim_{t\to\infty} t^{-2} \Psi_\mu(t) = -\infty$. It follows that $\mu$ cannot be quasi-infinitely divisible, for if it were, then $\lim_{t\to\infty} t^{-2} \Psi_\mu(t) = -A/2$ as shown in the proof of Lemma \ref{nnd}, where $A\in \bR$ denotes the Gaussian variance of $\mu$. Hence we have a one-dimensional distribution $\mu$ that is not quasi-infinitely divisible but whose characteristic function is zero-free. Using  Lemma \ref{projection} it then is easily seen that $\delta_0^{\otimes (d-1)} \otimes \mu$ is a distribution in $\bR^d$ that is not quasi-infinitely divisible but whose characteristic function is zero-free. Further examples of such distributions can be constructed using Proposition \ref{prop-indep} below.
\end{example}

Let us now give some concrete examples of quasi-infinitely divisible distributions on $\bZ^d$:

\begin{example}
(a)
Consider the distribution $\mu \coloneqq a \delta_{(0,0)} + b \delta_{(1,0)} + c \delta_{(0,1)}$ on $\bR^2$ with $a,b,c \in (0,1)$.
If $\max\{a,b,c\} > 1/2$, then $\mu$ is quasi-infinitely divisible by Cuppens' result (cf. Theorem \ref{cuppens}).
If $\max\{a,b,c\} \leq 1/2$, then $\mu$ cannot be quasi-infinitely divisible, since $\widehat{\mu}$ is not zero-free.
Indeed, for $(x,y) \in \bR^2$,
$ \widehat{\mu}(x,y)
= a + b e^{\ii x} + c e^{\ii y} = 0$
if and only if
$ b e^{\ii x}  = -a - c e^{\ii y}$.
The set $\{be^{\ii x}: x \in \bR\}$ describes a circle in the complex plane with center $0$ and radius $b$, intersecting the real axis at the points $-b$ and $b$, and $\{-a-ce^{\ii y}: y \in \bR\}$ describes a circle in the complex plane with center $-a$ and radius $c$, intersecting the real axis at the points $-a-c$ and $-a+c$.
Now, since $\max\{a,b,c\} \leq 1/2$, this implies that $a \leq b + c$, $b \leq a + c$ and $c\leq a+b$, and hence $-a-c \leq -b \leq -a+c\leq b$.
Therefore, the two circles intersect or touch each other, so they share at least one common point, which corresponds to a zero of the characteristic function of $\mu$.\\
(b)
Let $p, q \in (0,1) \setminus \{1/2\}$ and consider the distributions $\mu_1 \coloneqq (1-p) \delta_{(0,0)}+ p \delta_{(1,0)}$ and $\mu_2 \coloneqq (1-q) \delta_{(0,0)} + q \delta_{(0,1)}$.
Due to Theorem \ref{cuppens}, the distributions $\mu_1$ and $\mu_2$ are quasi-infinitely divisible.
Hence, also the  distribution
$ \mu \coloneqq \mu_1 \ast \mu_2 = (1-p)(1-q) \delta_{(0,0)} + p(1-q) \delta_{(1,0)} + (1-p)q \delta_{(0,1)} + pq \delta_{(1,1)}$ is quasi-infinitely divisible. Observe that it is possible to choose $p$ and $q$ such that $\max \{pq, (1-p) (1-q) , p(1-q), q(1-p)\} < 1/2$.\\
(c)
Let $\mu_1, \ldots, \mu_d$ be distributions on $\bR$ supported in $\bZ$ such that $\Re \, \widehat{\mu_k}(z) > 0$ for all  $k \in \{1, \ldots, d\}$ and $z \in \bR^d$.
Examples of such distributions can be obtained as  symmetrisations of distributions supported in $\bZ$, whose characteristic functions have no zeros.
Let $0 \leq p_1, \ldots, p_d \leq 1$ be such that $\sum_{k=1}^{d}p_k = 1$ and define the distribution
\begin{align*}
\mu \coloneqq \sum_{k=1}^{d} p_k \sum_{l\in \bZ} \mu_k(\{l\}) \delta_{le_k}
\end{align*}
on $\bR^d$, where $e_k$ is the k-th unit vector.
Then for $z =(z_1,\ldots, z_d)^T \in \bR^d$ we have
\begin{align*}
\widehat{\mu}(z)
&= \sum_{l \in \bZ^d} \mu(\{l\})e^{\ii \langle z,l \rangle}
= \sum_{k = 1}^d \sum_{l \in \bZ \setminus \{0\}} p_k \mu(\{le_k\})e^{\ii \langle z, le_k \rangle} + \mu(\{0\}) \\
&= \sum_{k = 1}^d p_k \sum_{l \in \bZ \setminus\{0\}} \mu_k(\{l\})e^{\ii l z_k}
+ \sum_{k=1}^{d} p_k \mu_k(\{0\})
= \sum_{k=1}^d p_k \widehat{\mu_k}(z_k) \neq 0
\end{align*}
since $\Re \, \widehat{\mu_k}(z_k) > 0$.
In particular, $\widehat{\mu}$ has no zeros, so $\mu$ is quasi-infinitely divisible by Theorem~\ref{gitter}.\\
(d)
Let $p \in [0, 1/4)$ and consider the symmetric distribution $\sigma \coloneqq p \delta_{-1} +(1-2p) \delta_0 + p \delta_1$.
We have $\widehat{\sigma}(z) = pe^{-\ii z} + 1-2p + pe^{\ii z} = 1-2p+2p\cos(z) > 0$ for $z \in \bR$ since $p < 1/4$.
With the construction in part (c), choosing $\mu_1 = \mu_2 = \sigma$ it follows that the distribution \begin{align*}
\mu
\coloneqq rp \delta_{(-1,0)} + rp \delta_{(1,0)} + (1-r)p \delta_{(0,-1)} + (1-r)p \delta_{(0,1)} + (1-2p) \delta_{(0,0)}
\end{align*}
is quasi-infinitely divisible for any $r\in [0,1]$.
\end{example}

Similar as in the one-dimensional case in \cite[Cor. 8.3]{lindner}, one can show that a consequence of Theorem~\ref{gitter} is that every factor of a quasi-infinitely divisible distribution that is supported in $\bZ^d$ is also quasi-infinitely divisible.

\begin{corollary}
Let $\mu$ be a quasi-infinitely divisible distribution supported in $\bZ^d$.
If $\mu_1$ and $\mu_2$ are distributions on $\bR^d$ such that $\mu = \mu_1 \ast \mu_2$, then also $\mu_1$ and $\mu_2$ are quasi-infinitely divisible.
\end{corollary}

\begin{proof}
 This follows in  complete analogy to the proof of \cite[Cor. 8.3]{lindner}, and is an easy consequence of Corollary \ref{c-lattice} and the fact that if $\mu$ is supported in $\bZ^d$, then there must be $k\in \bR^d$ such that $\mu_1$ is supported in $\bZ^d + k$ and $\mu_2$ is supported in $\bZ^d - k$.
\end{proof}

Theorem \ref{gitter} is nice since it gives a complete characterisation of quasi-infinite divisibility in terms of the characteristic function. In the univariate setting Berger \cite[Thm. 4.12]{berger} extended this characterisation of quasi-infinitely divisibility to a greater class of distributions, which we now state without proof:

\begin{theorem} \label{thm-berger}
Let $\mu$ be a distribution on $\bR$ of the form $\mu = \mu_d + \mu_{ac}$ with an absolutely continuous measure $\mu_{ac}$ and a non-zero discrete measure $\mu_d$ which is supported in the lattice $h\bZ + d$ for some $r \in \bR$, $h > 0$, such that $\widehat{\mu_d}(z) \neq 0$ for all $z \in \bR$.
Then $\mu$ is quasi-infinitely divisible if and only if $\widehat{\mu}(z) \neq 0$ for all $z \in \bR$. In that case, the Gaussian variance of $\mu$ is zero and the quasi-L\'evy measure $\nu$ satisfies $\int_{[-1,1]} |x|\, |\nu|(\di x) < \infty$.
\end{theorem}

It is remarkable that the quasi-L\'evy measure $\nu$ in Theorem \ref{thm-berger} can indeed be infinite although the distribution $\mu$ there has atoms; a concrete example for this phenomenon will be given in Example \ref{continuity} below. A special case of Theorem \ref{thm-berger} is when $\mu_d = p \delta_x$ for some $x\in \bR$ and $p>0$. Then $\widehat{\mu_d}(z) \neq 0$ for all $z\in \bR$ and a distribution of the form $\mu(\di x) = p \delta_x (\di x) + (1-p) f(x) \, \di x$, where $f:\bR \to [0,\infty)$ is integrable with integral~1, is quasi-infinitely divisible if and only if $\widehat{\mu}(z) \neq 0$ for all $z\in \bR$, cf. \cite[Thm. 4.6]{berger}. Further applications of Theorem \ref{thm-berger} include convex combinations of normal distributions: let $\mu = \sum_{i=1}^n p_i N(b_i,a_i)$, where $0<p_1,\ldots, p_n < 1$, $\sum_{i=1}^n p_i = 1$, $0<a_1 < a_2 < \ldots <a_n$ and $b_1,\ldots, b_n \in \bR$, where $N(b_i,a_i)$ denotes the normal distribution with mean $b_i$ and variance $a_i$. Then $\mu$ is quasi-infinitely divisible if and only if $\widehat{\mu}(z) \neq 0$ for all $z\in \bR$, as shown in \cite[Rem. 4.12]{berger}. Observe that the latter condition is in particular satisfied when additionally $b_1=\ldots = b_n=0$, and one can even show that $\sum_{i=1}^n p_i N(0,a_i)$ is quasi-infinitely divisible, even when some of the variances $a_i$ coincide, cf. \cite[Ex. 4.16]{berger}.

It seems likely that Theorem \ref{thm-berger} continues to hold in the multivariate setting, but so far we have not proved that. We intend to invest this case in future work. For the moment, we content ourselves with a recipe for constructing multivariate quasi-infinitely divisible distributions from independent one-dimensional quasi-infinitely divisible distributions.

\begin{proposition} \label{prop-indep}
Let $X_1, \ldots, X_d$ be independent real-valued random variables and let $X = (X_1, \ldots, X_d)^T$.
Then the law $\cL(X)$ of $X$ is quasi-infinitely divisible if and only if $\cL(X_k)$ is quasi-infinitely divisible for all $k \in \{1, \ldots, d\}$.
In this case, if $(A, \nu, \gamma)$ denotes the (standard) characteristic triplet of $\law(X)$ and $(a_k, \nu_k, \gamma_k)$ the (standard) characteristic triplet of $\law(X_k)$, then
\begin{align*}
A = \begin{pmatrix}
a_1 & 0 & \ldots & 0\\
0 & a_2 & \ddots & \vdots\\
\vdots & \ddots & \ddots & 0\\
0 & \ldots & 0& a_d
\end{pmatrix},
\quad
\gamma =
\begin{pmatrix}
\gamma_1 \\ \vdots \\ \gamma_d
\end{pmatrix}
\quad
\mbox{and} \quad
\nu = \sum_{k=1}^d  \delta_0^{\otimes (k-1)}  \otimes \nu_k \otimes \delta_0^{\otimes (d-k)} .
\end{align*}
\end{proposition}

\begin{proof}
That quasi-infinite divisibility of $\law(X)$ implies quasi-infinite divisibility of $\law(X_k)$ for $k=1,\ldots, d$ is clear from Lemma \ref{projection}. Conversely, let $\cL(X_k) \sim \qid(a_k,\nu_k, \gamma_k)$ for $k \in \{1, \ldots, d\}$ with independent $X_1,\ldots, X_d$. Using
$\widehat{\law(X)}(z)
= \prod_{k=1}^d \widehat{\law(X_k)}(z_k)$
for all $z = (z_1, \ldots, z_d)^T \in \bR^d$ it is easy to see that $\mu$ has a L\'evy--Khintchine type representation as in \eqref{lkf} with $A$, $\nu$ and $\gamma$ as given in the theorem, and hence that $\law(X)$ is quasi-infinitely divisible with standard characteristic triplet $(A,\nu,\gamma)$.
\end{proof}

Using Lemma \ref{projection}, the following is now immediate:

\begin{corollary}
Let $X = (X_1, \ldots, X_d)^T$ with independent real-valued random variables $X_1, \ldots, X_d$ such that $\cL(X_k)$ is quasi-infinitely divisible for each $k\in \{1,\ldots, d\}$.
Let further $n \in \bN$, $M \in \bR^{n \times d}$ and $b \in \bR^{n}$.
Then also the distribution $\cL(MX+b)$ is quasi-infinitely divisible.
\end{corollary}

\section{Conditions for absolute continuity} \label{S4}
\setcounter{equation}{0}

In this section we study absolute continuity of quasi-infinitely divisible distributions and give some sufficient conditions in terms of the characteristic triplet.
Considering an infinitely divisible distribution on $\bR$, Kallenberg \cite[p. 794 f.]{kallenberg} gave a sufficient condition on the L\'evy measure for the distribution to have a smooth Lebesgue density.
The following theorem generalizes this result for quasi-infinitely divisible distributions on $\bR^d$. In its statement, we denote by $S^{d-1} \coloneqq \{x \in \bR^d : |x|=1\}$ the unit sphere in $\bR^d$.

\begin{theorem} \label{abscont}
Let $\mu$ be a quasi-infinitely divisible distribution on $\mathbb{R}^d$ with characteristic triplet $(A, \nu, \gamma)$.
Define
\begin{align*}
G^-(r) \coloneqq \sup_{\xi \in S^{d-1}}\xi^T \left( \int_{|x|\leq r}xx^T \nu^-(\dx)\right) \xi
\quad \textrm{ and } \quad
G^+(r) \coloneqq  \inf_{\xi \in S^{d-1} } \xi^T \left( \int_{|x| \leq r} xx^T \nu^+(\dx)\right) \xi
\end{align*}
for $r >0$.
Suppose that $A$ is strictly positive definite, or that
\begin{align} \label{kallenberg}
\lim_{r \to 0} r^{-2} |\log r|^{-1}G^+(r) \left(
	\frac{1}{3} -
	2 \, \frac{r^2 \nu^-(\{x \in \bR^d: |x| > r\})}{ G^+(r)}
	-\frac{2}{3} \frac{G^-(r)}{G^+(r)}
\right) = \infty
\end{align}
(when $G^+(r) =0$ for small $r>0$ we interpret the left-hand side of \eqref{kallenberg} as 0 and hence \eqref{kallenberg} to be violated).
Then $\mu$ has an infinitely often differentiable Lebesgue density whose derivatives tend to $0$ as $|x| \to \infty$.
\end{theorem}

\begin{proof}
Suppose first that $A$ is strictly positive definite and let $\lambda_0 >0$ be the smallest eigenvalue of $A$.
By \cite[Lem. 43.11~(i)]{sato} we have
\begin{align*}
\lim_{|z| \to \infty} |z|^{-2}\int_{\bR^d} \left(e^{\ii \langle z,x \rangle} -1- \ii \langle z,x \rangle \1_{[0,1]}(|x|)\right) \nu^{\pm}(\dx) = 0.
\end{align*}
Since further $- \frac{1}{2}\langle z, A z \rangle
\leq - \frac{\lambda_0}{2} |z|^2 $
for all $z \in \bR^d$, we have $\limsup_{|z| \to \infty}|z|^{-2} \Psi_\mu(z) \leq -\frac{1}{2} \lambda_0,$ so there exists $K> 0$ such that $|\widehat{\mu}(z)| = e^{\Re \Psi_\mu(z)}\leq e^{-{\lambda_0|z|^2}/{4}}$ for all $z \in \bR^d$ with $|z| \geq K$.
Hence, we have $\int_{\bR^d} |\widehat{\mu}(z)| |z|^k \dz < \infty$ for all $k\in \bN$ and the claim follows by \cite[Prop. 28.1]{sato}.

Now, suppose that \eqref{kallenberg} is satisfied.
Using the fact that $\frac{1}{3}y^2 \leq 1-\cos(y) \leq \frac{2}{3} y^2$ for all $y \in [-1,1]$, we estimate
\begin{align*}
\int_{\bR^d} ( 1- \cos \langle z,x \rangle ) \, \nu^+(\dx)
&\geq \frac{1}{3} \int_{|x| \leq 1/|z|} \langle z,x \rangle^2 \, \nu^+(\dx)
= \frac{1}{3} \int_{|x| \leq 1/|z|} z^Txx^Tz \, \nu^+(\dx)\\
&\geq \frac{1}{3}|z|^2 G^+\left(1/|z|\right)
\end{align*}
and similarly
\begin{align*}
\int_{\bR^d} ( 1- \cos \langle z,x \rangle )\, \nu^-(\dx)
&\leq \frac{2}{3} \int_{|x| \leq 1/|z|} \langle z,x \rangle^2 \, \nu^-(\dx) + 2\nu^-(\{x \in \bR^d: |x|> 1/|z|\})\\
&\leq \frac{2}{3}|z|^2 G^-\left(1/|z| \right)
+ 2\nu^-(\{x \in \bR^d: |x|> 1/|z|\})
\end{align*}
for all $z \in \bR^d$.
Hence, for $|z|\geq 1$,
\begin{align*}
- (\log|z|)^{-1} \Re \Psi_\mu(z)
&= (\log|z|)^{-1}  \left( \frac12 \langle z, Az \rangle + \int_{\bR^d} ( 1- \cos \langle z,x \rangle ) \nu(\dx) \right) \\
&\geq (\log|z|)^{-1} \left(
	\int_{\bR^d} ( 1- \cos \langle z,x \rangle ) \nu(\dx)
\right)
\to \infty \textrm{ as } |z| \to \infty
\end{align*}
by assumption.
As a consequence, for every $k \in \mathbb{N}$ there exists $K > 0$ such that
$\Re \Psi_\mu(z) \leq -(k+2)\log|z| $ for all $z \in \bR^d$ with $|z| > K$, and therefore $|\widehat{\mu}(z)| = e^{\Re \Psi_\mu(z)}\leq e^{-(k+2)\log |z|} = |z|^{-(k+2)}$ when $|z| > K$.
This implies that $\int_{\bR^d} |\widehat{\mu}(z)| |z|^k \dz < \infty$ for every $k \in \bN$ and the claim follows.
\end{proof}

\begin{remark} \label{rem-ac}
{\rm
(a) Observe that for fixed $r\in (0,1)$, the matrices $\int_{|x|\leq r} x x^T \nu^{\pm} (\di x) \in \bR^{d\times d}$ are symmetric and non-negative definite and that $G^+(r)$ is the smallest eigenvalue of the matrix $\int_{|x|\leq r} xx^T\, \nu^+(\dx)$ and $G^-(r)$ is the largest eigenvalue of $\int_{|x|\leq r} xx^T\, \nu^-(\dx)$.\\
(b) If $\mu$ is infinitely divisible, then $\nu^-=0$ and hence $G^- (r) = 0$ and \eqref{kallenberg} reduces to
\begin{equation} \label{eq-kallenberg2}
\lim_{r\to 0} r^{-2} |\log r|^{-1}G^+(r) = +\infty.
\end{equation}
Hence, if an infinitely divisible distribution $\mu$ on $\bR^d$ with characteristic triplet $(A,\nu^+,\gamma)$ is such that $A$ is strictly positive definitive or such that \eqref{eq-kallenberg2} is satisfied, then $\mu$ has an infinitely often differentiable Lebesgue density whose derivatives tend to 0 as $|x| \to \infty$. In dimension $d=1$, this reduces to $A>0$ or $\lim_{r\to 0} r^{-2} |\log r|^{-1} \int_{|x|\leq r} |x|^2 \, \nu^+(\di x) = +\infty$, which is Kallenberg's classical condition \cite[p. 794 f.]{kallenberg}. The described multivariate generalisation of Kallenberg's condition  seems to be new even in the case of infinitely divisible distributions.\\
(c) Let $\mu\sim \qid (A,\nu,\gamma)$ in $\bR^d$. A sufficient condition for \eqref{kallenberg} to hold is that \eqref{eq-kallenberg2} is satisfied along with $G^-(r) = o(G^+(r))$ and $r^2 \nu^- (\{x \in \bR^d: |x|> r\} )= o(G^+(r))$ as $r\to 0$, where we used the \lq\lq little o\rq\rq~ Landau symbol notation.\\
(d) Let $\mu\sim \qid (A,\nu,\gamma)$ in $\bR^d$. Define
$$g^-(r) \coloneqq \int_{|x|\leq r} |x|^2 \, \nu^-(\di x) \quad \textrm{ and } \quad
g^+(r) \coloneqq \int_{|x|\leq r} |x|^2 \, \nu^+(\di x)$$
for $r>0$. Then
$$g^{\pm} (r) = \int_{|x|\leq r} \trace (x x^T) \, \nu^{\pm} (\di x) = \trace \left(\int_{|x|\leq r} x x^T \, \nu^{\pm} (\di x)\right).$$
Since the trace of a symmetric $d\times d$-matrix is the sum of its eigenvalues, we observe from~(a) that
$$G^-(r) \leq g^-(r) \leq d \, G^-(r) \quad \textrm{ and } \quad G^+(r) \leq g^+(r).$$
So we can conveniently bound $G^-(r)$ from below and above in terms of $g^-(r)$, in particular, if we replace $G^-(r)$ in \eqref{kallenberg} by $g^-(r)$ (and leave the rest unchanged, in particular we do not replace $G^+(r)$ by $g^+(r)$) then we also obtain a sufficient condition for $\mu$ to have an infinitely often differentiable Lebesgue density with derivatives tending to 0 as $|x|\to \infty$. A similar remark applies to (c) above, where we can replace $G^-(r)$ by $g^-(r)$ (but not $G^+(r)$ by $g^+(r)$).
}
\end{remark}


In \cite[Thm. 7.1]{lindner}, an Orey-type condition (cf. \cite[Prop. 28.3]{sato}) was given for absolute continuity of one-dimensional quasi-infinitely divisible distributions. We can now generalise this to the multivariate setting:

\begin{corollary}
Let $\mu$ be a quasi-infinitely divisible distribution on $\bR^d$ with characteristic triplet $(A,\nu,\gamma)$. With the notations of Theorem \ref{abscont}, suppose that $A$ is strictly positive definite or that there exists some $\beta\in (0,2)$ such that
\begin{align}\label{eq:1}
\liminf_{r \to 0} r^{-\beta} G^+(r) > \limsup_{r\to 0} r^{-\beta} G^-(r) = 0.
\end{align}
Then $\mu$ is absolutely continuous and its Lebesgue density $f$ is infinitely often differentiable with all its derivatives tending to $0$ as $|x| \to \infty$.
\end{corollary}

\begin{proof}
The case when $A$ is strictly positive definite is clear, so suppose that \eqref{eq:1} is satisfied with some $\beta\in (0,2)$.
Then clearly $\lim_{r \to 0} r^{-2} |\log r|^{-1}G^+(r) = \infty$ since $\beta < 2$, and $G^-(r) = o(G^+(r))$ as $r\to 0$. By Remark \ref{rem-ac}~(c) it is hence sufficient to show that $r^2 \nu^- (\{x \in \bR^d: |x|> r\}) = o(G^+(r))$ as $r\to 0$. To see this, define $g^-(r) \coloneqq \int_{|x|\leq r} |x|^2 \, \nu^-(\di x)$. Using integration by parts and $g^-(r) \leq d \, G^-(r)$ by Remark \ref{rem-ac}~(d), we then obtain
\begin{align*}
 \nu^-(\{x \in \bR^d:r <  |x| \leq 1\})
&=\int_{r}^1 s^{-2} \, \textrm{d}g^-(s)\\
& =g^-(1) -r^{-2}g^-(r) - \int_{r}^1 g^-(s)\,  \textrm{d}s^{-2}\\
&\leq d \left( G^-(1) + 2 \int_{r}^1 G^-(s) s^{-3}  \, \textrm{d}s\right)
\end{align*}
for $r\in (0,1]$. By \eqref{eq:1} for every $\varepsilon > 0$ we can find an $r_\varepsilon \in (0,1)$ such that
$G^-(s) \leq \varepsilon s^\beta$ for all $s\in (0,r_\varepsilon]$, so that we continue to estimate for $r<r_\varepsilon$
$$ \nu^-(\{x \in \bR^d:r <  |x| \leq 1\}) \leq d G^-(1) + 2 d G^-(1) r_\varepsilon^{-3}(1-r_\varepsilon) + 2 \varepsilon  d \int_r^{r_\varepsilon} s^{\beta-3} \, \di s .$$
This implies
$$r^2 \frac{\nu^-(\{x \in \bR^d:r <  |x| \leq 1\})}{G^+(r)} \leq r^2 d \frac{G^-(1) + 2 G^-(1) r_\varepsilon^{-3} }{G^+(r)} + 2\varepsilon d \frac{r^{\beta}}{(2-\beta) G^+(r)}. $$
Denoting the limit inferior on the left hand side of \eqref{eq:1} by $L$, and observing that $\lim_{r\to 0} r^{-2} G^+(r) = \infty$ by \eqref{eq:1}, we obtain
$$\limsup_{r\to 0} \left( r^2 \nu^-(\{ x\in \bR^d : r < |x|\leq 1\}) / G^+(r) \right) \leq \frac{2\varepsilon d}{(2-\beta) L},$$ and since $\varepsilon > 0$ was arbitrary we see $\lim_{r\to 0} r^2 \nu^-(\{ x\in \bR^d : r < |x|\leq 1\}) / G^+(r) = 0$. That $\lim_{r\to 0} r^2 \nu^-(\{ x\in \bR^d : |x| > 1\}) / G^+(r) = 0$ is clear so that we obtain $r^2 \nu^- (\{x \in \bR^d: |x|> r\}) = o(G^+(r))$ as $r\to 0$, finishing the proof.
\end{proof}

\begin{example}
(a) Let $\mu$ be a non-trivial strictly $\alpha$-stable rotation invariant distribution on $\bR^d$, where $\alpha \in (0,2)$.
It is well-known that $\mu$ has a $C^{\infty}$-density with all derivatives vanishing at infinity.
Let us check that this can also be derived from Theorem \ref{abscont}.
The L\'evy measure $\nu$ of $\mu$ is given by $\nu(\dx) = C|x|^{-(d+\alpha)} \dx$ for some constant $C > 0$, see \cite[Ex. 62.1]{sato}.
For $r >0$ let $G^+(r) \coloneqq \inf_{\xi \in S^{d-1} } \xi^T\int_{|x| \leq r} xx^T \nu(\dx)\xi$.
Since  $\mu$ is infinitely divisible, condition \eqref{kallenberg} of Theorem \ref{abscont} reduces to \eqref{eq-kallenberg2}.
In order to show that \eqref{eq-kallenberg2} is satisfied, let $r > 0$.
For $k, j \in \{1, \ldots, d\}$, $k \neq j$ we obtain
\begin{align*}
\int_{|x|\leq r} x_k x_j \nu(\dx)
= C \int_{|x|\leq r} \frac{x_kx_j}{|x|^{d+\alpha}} \dx = 0.
\end{align*}
To see that, by symmetry it suffices to consider the case $k = {d}$.
Then we have
\begin{align*}
\int_{|x|\leq r} \frac{x_dx_j}{|x|^{d+\alpha}} \dx
= \int_{-r}^r x_d \int\limits_{\substack{x' \in \bR^{d-1}: \\ |x'| \leq \sqrt{r^2-x_d^2}}} \frac{x_j}{\left(x_d^2+(x')^2\right)^{(d+\alpha)/2}} \dx' \dx_d,
\end{align*}
and the integrand of the outer integral is an odd function.
Hence, the matrix $A_r \coloneqq \int_{|x|\leq r}xx^T \nu(\dx)$ is a diagonal matrix.
We compute the trace of $A_r$ as
\begin{align*}
\sum_{k=1}^{d} \int_{|x|\leq r} x_k^2 \,\nu(\dx)
&= C \int_{|x|\leq r}|x|^{2-\alpha-d} \dx
= C \int_0^r\int_{sS^{d-1}} s^{2-\alpha-d} \textnormal{d}\theta \textnormal{d}s\\
&= C \int_{0}^r s^{2-\alpha-d} \, s^{d-1} \omega_d \textnormal{d}s
 = C\omega_d r^{2 - \alpha}/(2-\alpha),
\end{align*}
where $\omega_d$ denotes the $(d-1)$-dimensional volume of the surface $S^{d-1}$.
Again by symmetry, it follows that $A_r = \frac{C\omega_d}{d(2-\alpha)} r^{2-\alpha} I_d$ with the identity matrix $I_d \in \bR^{d \times d}$.
Therefore, $G^+(r)= \frac{C\omega_d}{d(2-\alpha)} r^{2-\alpha}$
which implies that $\lim_{r \to 0} r^{-2}|\log r|^{-1} G^+(r) = \infty$,
showing that $\mu$ satisfies~\eqref{eq-kallenberg2}.\\
(b) Now let $\mu$ be a non-trivial rotation invariant strictly $\alpha$-stable distribution on $\bR^d$ as in (a), and let $\sigma$ be a probability distribution on $\bR^d$. Since $\mu$ has a $C^\infty$ density with all derivatives tending to $0$ as $|x|\to \infty$, the same is true for the convolution $\mu' := \mu \ast \sigma$. When $\sigma$ is additionally quasi-infinitely divisible and concentrated in $\bZ^d$, then this can be also seen from Theorem \ref{abscont}. To see this, observe that $\sigma$ has finite quasi-L\'evy measure $\nu_\sigma$ concentrated in $\bZ^d$ by Theorem \ref{gitter}. It follows that $\mu'$ has quasi-L\'evy measure $\nu_\sigma(\di x) + C |x|^{-(d+\alpha)} \, \di x$. Hence the quantities $G^{\pm}(r)$ for $\mu$ and $\mu'$ coincide when $r<1$ (with $G^-(r)$ being zero when $r<1$). It follows that also $\mu'$ satisfies the assumptions of Theorem \ref{abscont}, so that $\mu'$ has a $C^\infty$-density with derivatives tending to 0 as $|x|\to \infty$. This is of course a constructed example, but it shows that there are cases of quasi-infinitely divisible distributions that are not infinitely divisible for which the assumptions of Theorem~\ref{abscont} are applicable.
\end{example}

\begin{remark} \label{continuity}
{\rm
While the problem of a complete description of absolute continuity in terms of the L\'evy measure remains  challenging for infinitely divisible distributions, the corresponding question for continuity is completely solved: It is well-known that an infinitely divisible distribution $\mu$ on $\bR^d$ with characteristic triplet $(A, \nu, \gamma)$ is continuous if and only if $A \neq 0$ or $\nu(\bR) = \infty$, see \cite[Theorem 27.4]{sato}. The same characterisation fails however when considering quasi-infinitely divisible distributions. Berger \cite[Ex. 4.7]{berger} showed that the distribution $\mu = 0.001 \delta_0 + 0.999 N(1,1)$, where $N(1,1)$ is the one-dimensional normal distribution with mean and variance 1, is quasi-infinitely divisible with infinite quasi-L\'evy measure. Observe that $\mu$ is not continuous. Using Proposition \ref{prop-indep} it is then also easy to construct non-continuous multivariate quasi-infinitely divisible distributions with infinite quasi-L\'evy measure.}
\end{remark}

\section{Topological properties of the class of infinitely divisible distributions} \label{S5}
\setcounter{equation}{0}

Let $\QID(\bR^d)$ denote the set of all quasi-infinitely divisible distributions on $\bR^d$ and $\cP(\bR^d)$ the set of all distributions on $\bR^d$.
Equipped with the Prokhorov-metric $\pi$, $\cP(\bR^d)$ gets a metric space and the convergence in this space corresponds to the weak convergence of distributions.
In this section we will always identify $\cP(\bR^d)$ with this metric space.
The aim of this section is to state some topological properties of $\QID(\bR^d)$ that were already given by \cite{berger} and  \cite{lindner} in one dimension.
We start with some results that can be shown similarly in any dimension $d \in \bN$. The following was shown in \cite[Prop. 5.1]{berger} and \cite[Sect.~4]{lindner} in dimension 1.

\begin{theorem}
Let $d\in \bN$. The set $\QID(\bR^d)$ is neither open nor closed in $\cP(\bR^d)$.
Moreover, the set $\cP(\bR^d) \setminus \QID(\bR^d)$ is dense in $\cP(\bR^d)$.
\end{theorem}

\begin{proof}
To see that $\cP(\bR^d) \setminus \QID(\bR^d)$ is dense in $\cP(\bR^d)$, let $\mu$ be an arbitrary distribution on $\bR^d$ and let $\sigma$ be a distribution on $\bR^d$ such that its characteristic function $\widehat{\sigma}$ has zeros.
For $n \in \bN$ define the distribution $\mu_n$ by
$\mu_n(\dx) = \mu(\dx) \ast \sigma(n \dx)$.
The characteristic function of $\mu_n$ is given by $\widehat{\mu_n}(z) = \widehat{\mu}(z) \widehat{\sigma}(z/n)$ for $z \in \bR^d$ and has zeros, hence $\mu_n$ cannot be quasi-infinitely divisible.
Furthermore, $\widehat{\mu_n}(z) \to \widehat{\mu}(z)$ as $n \to \infty$ for every $z \in \bR^d$, implying that $\mu_n \xrightarrow{w} \mu$ as $n \to \infty$.
Hence,  $\cP(\bR^d) \setminus \QID(\bR^d)$ is dense in $\cP(\bR^d)$, so  $\cP(\bR^d) \setminus \QID(\bR^d)$ cannot be closed, therefore $\QID(\bR^d)$ cannot be an open set.
In order to show that $\QID(\bR^d)$ is not closed, first observe that for $n \in \bN$ the distribution $ \frac{n+1}{2n}\delta_0 + \frac{n-1}{2n} \delta_{e_1}$ is quasi-infinitely divisible due to Theorem~\ref{cuppens}, where $e_1 = (1,0,\ldots,0)^T \in \bR^d$ is the first unit vector.
We have
$$ \frac{n+1}{2n} \delta_0 + \frac{n-1}{2n}\delta_{e_1} \xrightarrow{w} \frac{1}{2} \delta_0 + \frac{1}{2} \delta_{e_1}
\quad \textrm{as } n \to \infty$$
and $\frac{1}{2} \delta_0 + \frac{1}{2} \delta_{e_1}$ is not  quasi-infinitely divisible since its characteristic function has zeros.
Hence, the set $\QID(\bR^d)$ cannot be closed.
\end{proof}

In any topological space and hence in any metric space one can define the notions of connected and path-connected subsets. Observe that path-connectedness implies connectedness (see \cite[Thm. 3.29]{Armstrong}). The following result shows that $\QID(\bR^d)$ is path-connected (with respect to the Prokhorov metric) and hence connected, generalising \cite[Prop. 5.2]{berger} for $\QID(\bR)$ to arbitrary dimensions.

\begin{theorem}
Let $d\in \bN$. The set $\QID(\bR^d)$ is path-connected and hence connected in $\cP(\bR^d)$.
\end{theorem}

\begin{proof}
Suppose that $\mu_0$ and $\mu_1$ are quasi-infinitely divisible distributions on $\bR^d$.
For $t \in (0,1)$ the distributions $\sigma_t^0$ and $\sigma_t^1$ defined by $\sigma_t^0(\dx) \coloneqq \mu_0(1/(1-t) \dx)$ and $\sigma_t^1(\dx) \coloneqq \mu_1(1/t \dx)$ are also quasi-infinitely divisible by Lemma \ref{projection}.
Therefore, also the distribution $\mu_t$ defined by
$$\mu_t(\dx) \coloneqq \mu_0\left( \frac{1}{1-t} \di x \right) \ast \mu_1\left( \frac{1}{t} \dx \right)$$
is quasi-infinitely divisible for every $t \in (0,1)$.
Note that $\widehat{\mu_t}(z) = \widehat{\mu_0}((1-t)z)\widehat{\mu_1}(tz)$ for $z \in \bR^d$.
The mapping
$$p: [0,1] \to \cP(\bR^d), \quad t \mapsto \mu_t$$ is continuous, because $\widehat{\mu_s}(z) \to \widehat{\mu_t}(z)$ as $s \to t$ for all $z \in \bR^d$ and hence $\mu_s \xrightarrow{w} \mu_t$ as $s \to t$.
Since $p(0) = \mu_0$ and $p(1) = \mu_1$, it follows that $\QID(\bR^d)$ is path-connected and hence connected.
\end{proof}

It is not surprising that $\cP(\bR^d) \setminus \QID (\bR^d)$ is dense in $\cP(\bR^d)$. Much more surprising is the fact that also $\QID(\bR)$ is dense in dimension 1. This was proved in \cite[Thm. 4.1]{lindner}. We state the precise result and also give a sketch of the proof along the lines of \cite{lindner}, in order to discuss afterwards where the obstacles arise when trying to generalise the result to higher dimensions.

\begin{theorem} \label{thm-dense}
The set of quasi-infinitely divisible distributions on $\bR$ with Gaussian variance zero and finite quasi-L\'evy measure is dense in $\cP(\bR)$. 
\end{theorem}

\emph{Sketch of proof.}
Denote by $\mathcal{Q}(\bR)$ the set of all probability measures  on $\bR$ whose support is a finite set contained in a lattice of the form $n^{-1} \bZ$ for some $n\in \bN$, and by $\QID^0(\bR)$ the set of all quasi-infinitely divisible distributions on $\bR$ with Gaussian variance 0 and finite quasi-L\'evy measure. It is easily seen that $\mathcal{Q}(\bR)$ is dense in $\cP(\bR)$. Hence, it suffices to show that $\mathcal{Q}(\bR) \cap \QID^0(\bR)$ is dense in $\mathcal{Q}(\bR)$. To show this, let $\mu$ be a  distribution in $\mathcal{Q}(\bR)$, say $\mu = \sum_{k=-m}^m p_k \delta_{k/n}$ for some $m,n\in \bN$ with $0\leq p_k \leq 1$ for all $k\in \{-m,\ldots, m\}$ and $\sum_{k=-m}^m p_k = 1$, and let $X$ be a random variable with distribution $\mu$. From Lemma~\ref{projection} it is clear that $\mu=\law(X)$ is in the closure of $\mathcal{Q}(\bR) \cap \QID^0(\bR)$ if and only if $\law(n X+m)$ is in the closure of $\mathcal{Q}(\bR) \cap \QID^0(\bR)$. Hence it is sufficient to consider distributions $\mu$ whose support is a finite set contained in $\bN_0$. If the support of $\mu$ is contained in $\{0,\ldots, m\}$ for some $m\in \bN_0$, then it is easily seen that $\mu$ can be approximated arbitrarily well with distributions having support exactly $\{0,\ldots, m\}$, i.e. distributions $\sigma$ of the form $\sigma= \sum_{k=0}^m p_k \delta_k$ with $p_0,\ldots, p_m>0$ (strictly positive). So it is sufficient to show that any such distribution $\sigma$ can be approximated arbitrarily well by distributions in $\mathcal{Q}(\bR) \cap \QID^0(\bR)$.
If the characteristic function $\widehat{\sigma}$ of $\sigma$ has no zeroes, then we are done by Theorem \ref{gitter}.
So suppose that $\bR \ni z \mapsto \widehat{\sigma}(z) = \sum_{k=0}^m p_k \re^{\ri z k}$  has zeroes, which corresponds to zeroes of the polynomial $f$ given by $f(\omega) = \sum_{k=0}^{m} p_k\omega^k$ for $\omega \in \bC$ on the unit circle.
Factorising $f$, we can write $f(\omega) = p_m \prod_{k=1}^{m}(\omega-\xi_k)$ for $\omega \in \bC$ with $\xi_k \in \bC$ for $k \in \{1,\ldots,m\}$.
For $h > 0$ let
\begin{align} \label{eq-modified}
	f_h(\omega)
	\coloneqq p_{m}\prod_{k=1}^{m}(\omega-\xi_k-h)
	\quad \textrm{for all } \omega \in \bC.
\end{align}
If $h$ is chosen small enough, then $f_h$ has no zeros on the unit circle.
The polynomial $f$ has real coefficients, so the non-real roots of $f$ appear in pairs of complex conjugates.
By construction, the same is true for $f_h$, so there exist $\alpha_{h,k} \in \bR$ for $k \in \{0,\ldots,m\}$ such that $f_h(\omega) = \sum_{k=0}^{m}\alpha_{h,k}\omega^k$ for all $\omega \in \bC$.
Moreover, $\alpha_{h,k} \to p_k>0$ as $h \to 0$, so  we can assume that $h$ is small enough such that $\alpha_{h,k} > 0$ for all $k \in \{0,\ldots,m\}$.
Let $\sigma_h \coloneqq \lambda^{-1} \sum_{k=1}^{m} \alpha_{h,k} \delta_k$ with $\lambda \coloneqq \sum_{k=0}^{m}\alpha_{h,k}$. Then for small enough $h>0$, $\sigma_h$ is a probability distribution having support $\{0,\ldots, m\}$ and the characteristic function of $\sigma_h$ has no zeroes. By Theorem \ref{gitter},  $\sigma_h$ is quasi-infinitely divisible with finite quasi-L\'evy measure and Gaussian variance $0$, and $\sigma_h \stackrel{w}{\to} \sigma$ as $h\downarrow 0$.
$\Box$

It is very tempting now to assume that Theorem \ref{thm-dense} also holds for $\bR^d$-valued distributions with general $d\in \bN$. Again it is easily seen that it would suffice to show that any distribution $\sigma = \sum_{k_1,\ldots, k_d=0}^m p_{(k_1,\ldots,k_d)} \delta_{(k_1,\ldots, k_d)}$ with $p_{k_1,\ldots, k_d}>0$ can be approximated arbitrarily well by quasi-infinitely distributions in $\bZ^d$. Since a distribution in $\bZ^d$ is quasi-infinitely divisible if and only if its characteristic function has no zeroes, this might appear to be an easy task at first glance. However, it is not clear how to do a modification as in Equation~\eqref{eq-modified}, the problem being that polynomials in more than one variable do not factorise and also that there may be infinitely many zeroes of such polynomials. We have tried some time to pursue such a path, but have not succeeded. Having failed in proving Theorem~\ref{thm-dense} for dimensions $d\geq 2$, it is natural to wonder whether such a generalisation may be true at all. Let us pose this as an open question:

\begin{question}
Is $\QID(\bR^d)$ dense in $\cP(\bR^d)$ for any dimension $d\in \bN$, or is it dense only for $d=1$, or is it dense for certain dimensions and not for others?
\end{question}

Passeggeri \cite[Conjecture 4.1]{Passeggeri2020b} conjectures that $\QID(\bR^d)$ is dense in $\cP(\bR^d)$ for all dimensions $d\in \bN$. It is possible that this is true, but we are more inclined to believe that $\QID(\bR^d)$ will be dense in $\cP(\bR^d)$ if and only if $d=1$. An indication that this might be the case is the fact that the set of all zero-free complex valued continuous functions on $[0,1]^d$ is dense in the set of all continuous complex valued functions on $[0,1]^d$ with respect to uniform convergence if and only if $d=1$, see Pears \cite[Prop. 3.3.2]{Pears}. Knowing this, we wonder if even the set of all probability distributions on $\bR^d$ having zero-free characteristic function is dense in $\cP(\bR^d)$ if and only if $d=1$, but again we do not know the answer. We intend to invest this question more deeply in the future.

\section{Conditions for weak convergence} \label{S6}
\setcounter{equation}{0}

Weak convergence of infinitely divisible distributions can be characterised in terms of the characteristic triplet, see Sato \cite[Thms. 8.7, 56.1]{sato}.
This characterisation cannot be fully generalised to quasi-infinitely divisible distributions, since the class of quasi-infinitely divisible distributions is not closed with respect to weak convergence. In \cite[Thm. 4.3]{lindner}, some sufficient conditions for weak convergence of one-dimensional quasi-infinitely divisible distributions in terms of the characteristic pair where given; the \emph{characteristic pair} of a one-dimensional quasi-infinitely divisible distribution with characteristic triplet $(A,\nu,\gamma)_c$ with respect to a representation function $c:\bR \to \bR$ is given by $(\zeta,\gamma)_c$, where $\zeta$ is a finite signed measure on $\bR$ given by $\zeta(\di x) = A \, \delta_0(\di x) + (1\wedge x^2) \, \, \nu(\di x)$. It is not so easy to generalise the characteristic pair to the multivariate setting and hence we will rather work with another characterisation of weak convergence, in line with the conditions given in \cite[Thm. 56.1]{sato} for weak convergence of infinitely divisible distributions.

Denote by $C_{\#}$ the set of all bounded, continuous functions $f: \bR^d \to \bR$ vanishing on a neighborhood of $0$. Then the following provides a sufficient condition for weak convergence of quasi-infinitely divisible distributions in terms of the characteristic triplets.

\begin{theorem}
Let $c: \bR^d \to \bR^d$ be a continuous representation function.
Let $(\mu_n)_{n \in \bN}$ be a sequence of quasi-infinitely divisible distributions on $\bR^d$ such that $\mu_n$ has characteristic triplet $(A_n, \nu_n, \gamma_n)_c$ for every $n \in \bN$ and let $\mu$ be a quasi-infinitely divisible distribution on $\bR^d$ with characteristic triplet $(A, \nu, \gamma)_c$.
Suppose that the following conditions are satisfied.
\begin{enumerate}
\item For all $f \in C_\#$ it holds
\begin{align*}
\lim_{n \to \infty} \int_{\bR^d} f(x) \nu_n^\pm(\dx)
= \int_{\bR^d} f(x) \nu^\pm (\dx).
\end{align*}
\item If $A_{n, \varepsilon}$ is defined by
\begin{align*}
\langle z, A_{n, \varepsilon} z \rangle
= \langle z, Az \rangle
	+ \int_{|x| \leq \varepsilon} \langle z, x \rangle^2 \nu_n^+(\dx)
\end{align*}
for all $n \in \bN$ and $\varepsilon > 0$, then
\begin{align*}
\lim_{\varepsilon \downarrow 0} \limsup_{n \to \infty} | \langle z, A_{n, \varepsilon} z \rangle - \langle z, Az \rangle| = 0
\quad \textrm{for all } z \in \bR^d.
\end{align*}
\item It holds
\begin{align*}
\lim_{\varepsilon \downarrow 0} \limsup_{n \to \infty} \int_{|x| \leq \varepsilon} |x|^2 \nu_n^-(\dx) = 0.
\end{align*}
\item $\gamma_n \to \gamma$ as $n \to \infty$.
\end{enumerate}
Then $\mu_n \xrightarrow{w} \mu$ as $n \to \infty$.
\end{theorem}

\begin{proof}
Let $\mu^1$ and $\mu^2$ be infinitely divisible distributions with characteristic triplets $(A, \nu^+, \gamma)_c$ and $(0, \nu^-,0)_c$ and for $n \in \bN$ let $\mu_n^1$ and $\mu_n^2$ be infinitely divisible distributions with characteristic triplets $(A_n, \nu_n^+, \gamma_n)_c$ and $(0, \nu_n^-,0)_c$, respectively.
Then $\mu^2 \ast \mu = \mu^1$ and $\mu_n^2 \ast \mu_n = \mu_n^1$ for every $n \in \bN$.
Further, $\mu_n^1 \xrightarrow{w} \mu^1$ and $\mu_n^2 \xrightarrow{w} \mu^2$ by \cite[Thm. 56.1]{sato} and hence $\mu_n \xrightarrow{w} \mu$ as $n \to \infty$ which follows by considering characteristic functions.
\end{proof}

The next result presents a sufficient condition for a weak limit of quasi-infinitely divisible distributions to be again quasi-infinitely divisible:

\begin{theorem}
Let $c: \bR^d \to \bR^d$ be a continuous representation function and for $n \in \bN$ let $\mu_n$ be a quasi-infinitely divisible distribution on $\bR^d$ with characteristic triplet $(A_n, \nu_n, \gamma_n)_c$.
Suppose that the sequence $(\mu_n)_{n \in \bN}$ converges weakly to some distribution $\mu$ and that there exists a L\'evy measure $\sigma$ on $\bR^d$ such that
\begin{align*}
\lim_{n \to \infty}\int_{\bR^d} f(x) \nu_n^-(\dx)
= \int_{\bR^d} f(x) \sigma(\dx)
\end{align*}
for all $f \in C_{\#}$ and
\begin{align*}
\lim_{\varepsilon \downarrow 0} \limsup_{n \to \infty}
\int_{|x| \leq \varepsilon} |x|^2 \nu_n^-(\dx) = 0.
\end{align*}
Then $\mu$ is quasi-infinitely divisible.
If we denote its characteristic triplet by $(A, \nu, \gamma)_c$, then $\nu+\sigma$ is a L\'evy measure,
\begin{align*}
\lim_{n \to \infty}\int_{\bR^d} f(x) \nu_n^+(\dx)
= \int_{\bR^d} f(x) (\nu+\sigma)(\dx),
\end{align*}
 $\gamma_n \to \gamma$ as $n \to \infty$ and for $A_{n, \varepsilon} \in \bR^{d \times d}$ defined by
\begin{align*}
\langle z, A_{n, \varepsilon} z \rangle
= \langle z, Az \rangle
	+ \int_{|x| \leq \varepsilon} \langle z, x \rangle^2 \nu_n^+(\dx)
\end{align*}
for all $n \in \bN$ and $\varepsilon > 0$ we have
\begin{align*}
\lim_{\varepsilon \downarrow 0} \limsup_{n \to \infty} | \langle z, A_{n, \varepsilon} z \rangle - \langle z, Az \rangle| = 0
\quad \textrm{for all } z \in \bR^d.
\end{align*}
\end{theorem}

\begin{proof}
Let $\mu_n^1$ and $\mu_n^2$ be infinitely divisible distributions with characteristic triplets $(A_n, \nu_n^+, \gamma_n)_c$ and $(0, \nu_n^-,0)_c$, respectively, and  $\mu^2$ be infinitely divisible with characteristic triplet $(0, \sigma, 0)_c$.
Then \cite[Thm. 56.1]{sato} implies $\mu_n^2 \xrightarrow{w} \mu^2$, and hence also $\mu_n^1 = \mu_n^2 \ast \mu_n \xrightarrow{w} \mu^2 \ast \mu$ as $n\to\infty$.
Since $\mu_n^1$ is infinitely divisible for each $n\in \bN$, so is $\mu^2 \ast \mu$.
If we denote its characteristic triplet by $(A, \eta, \gamma)$, then $\mu$ is quasi infinitely divisible with characteristic triplet $(A, \nu, \gamma)_c$, where $\nu = \eta_{|\cB_0} - \sigma_{|\cB_0}$.
The other implications now follow from \cite[Thm. 56.1]{sato}.
\end{proof}

\section{Support properties} \label{S7}
\setcounter{equation}{0}

For infinitely divisible distributions on $\bR$, the boundedness of the support from below can be characterised in terms of the characteristic triplet.
More precisely, an infinitely divisible distribution $\mu$ with characteristic triplet $(a, \nu, \gamma)$ has support bounded from below if and only if $a=0$ and $\nu$ is supported in $[0, \infty)$, c.f. Sato \cite[Thm. 24.7]{sato}. If $\mu$ is only quasi-infinitely divisible (and not necessarily infinitely divisible), then it was shown in \cite[Prop. 5.1]{lindner} that the following two statements (i) and (ii) are equivalent:
\begin{enumerate}
\item[(i)] $\mu$ is bounded from below, $\supp (\nu^-) \subset [0,\infty)$  and $\int_{|x|\leq 1} |x| \, \nu^-(\di x) < \infty$.
\item[(ii)] $a=0$, $\supp (\nu^+) \subset [0,\infty)$ and   $\int_{|x|\leq 1} |x| \nu^+(\dx) < \infty$.
\end{enumerate}
Observe that for infinitely divisible distributions we have $\nu^- = 0$ and $\nu^+ = \nu$, so that the above result reduces to the known characterisation for infinitely divisible distributions. Also observe that the condition \lq\lq $\mu$ is bounded from below\rq\rq~can be rewritten as \lq\lq there is $b\in \bR$ such that $b + \supp (\mu) \subset [0,\infty)$\rq\rq.   Our goal now is to extend this result to higher dimensions, and we will be working immediately with cones rather than only $[0,\infty)^d$. Following \cite[Def. 4.8]{Rocha-Sato}, by a \emph{cone} in $\bR^d$ we mean a non-empty closed convex subset $K$ of $\bR^d$ which is not $\{0\}$, does not contain a straight line through the origin and is such that with $x\in K$ and $\lambda \geq 0$ also $\lambda x\in K$. Observe that this definition is more restrictive than the usual notion of cones in linear algebra, in the sense that we require additionally a cone to be closed, convex, non-trivial and one-sided (the latter being sometimes also called \lq\lq proper\rq\rq), but in probability this seems to be more standard and they are the only cones of interest to us. Obviously, $[0,\infty)^d$ is a cone, but there are many other examples.

Coming back to the question when a quasi-infinitely divisible distribution has support contained in a translate of a cone, let us first recall the corresponding results for infinitely divisible distributions; the equivalence of (i) and (iii) below can be found in Skorohod \cite[Thm. 21 in \S 3.3]{skorohod}  or Rocha-Arteaga and Sato \cite[Thm. 4.11]{Rocha-Sato}, while  the equivalence of (ii) and (iii) is stated in Equation (4.35) of \cite{Rocha-Sato}; since the latter equivalence is only given in a remark in \cite{Rocha-Sato} we provide a short sketch of the proof for this fact.

\begin{theorem} \label{thm-rocha}
Let $L=(L_t)_{t\geq 0}$ be a L\'evy process in $\bR^d$ with characteristic triplet $(A,\nu,\gamma)$ (i.e. $\law(L_1)$ has this characteristic triplet) and let $K\subset \bR^d$ be a cone. Then the following are equivalent:
\begin{enumerate}
\item[{\rm (i)}] $\supp (\law(L_t)) \subset K$ for every $t\geq 0$.
\item[{\rm (ii)}] $\supp (\law(L_t)) \subset K$ for some $t>0$.
\item[{\rm (iii)}] $A=0$, $\supp (\nu) \subset K$, $\int_{|x|\leq 1} |x|\, \nu(\di x) < \infty$ and $\gamma^0 \in K$, where $\gamma^0$ is the drift of $\law(L_1)$.
\end{enumerate}
\end{theorem}

\emph{Sketch of proof of the equivalence of (i) and (ii).} That (i) implies (ii) is clear. For the converse, assume that $\supp (\law(L_t)) \subset K$ for some $t>0$. Since $\law(L_t) = (\law(L_{t/2}))^{\ast 2}$ we also have $\supp (\law(L_{t/2})) \subset K$; to see that, suppose there were $y\in  \supp (\law(L_{t/2})) \setminus K$. Then also $2y \notin K$ and by the closedness of $K$ there is an open ball $U$ containing $y$ with $U \cap K = \emptyset$, $(U+U)\cap K= \emptyset$ and $P(L_{t/2} \in U) > 0$. Then also $P(L_t \in U+U) > 0$, a contradiction. Hence we have $\supp (L_{t/2}) \subset K$ and iterating this argument we obtain $\supp (\law(L_{2^{-n} t})) \subset K$ for all $n\in \bN_0$. Since $K+K\subset K$ and $K$ is closed we conclude $\supp (\law (L_{q 2^{-n} t})) \subset K$ for all $q,n\in \bN_0$. Since $\{q 2^{-n} t: q,n\in \bN_0\}$ is dense in $[0,\infty)$ we get~(i) since $L$ has right-continuous sample paths. $\Box$

We also need the following easy lemma for infinitely divisible distributions with existing drift. It is well known, but since we were unable to find a ready reference, we include a short proof.

\begin{lemma} \label{lem-drift}
Let $\mu$ be an infinitely divisible distribution on $\bR^d$ with characteristic triplet $(A,\nu,\gamma)$ such that $\int_{|x|\leq 1} |x| \, \nu(\di x) < \infty$. Then the drift of $\mu$ is an element of $\supp(\mu)$.
\end{lemma}

\begin{proof}
Let $L=(L_t)_{t\geq 0}$ be a L\'evy process such that $\mu = \law(L_1)$. By the L\'evy--It\^o decomposition we can write $L_t = B_t + \gamma^0 t + \sum_{0<s\leq t} \Delta L_s$ for each $t\geq 0$, where $B=(B_t)_{t\geq 0}$ is a Brownian motion with drift 0 and covariance matrix $A$, $\gamma^0$ denotes the drift of $L$ (i.e. of $\law(L_1)$) and $\Delta L_s$ denotes the jump size of $L$ at $s$. Since the three components are independent and 0 is obviously in the support of both $\law(B_t)$ and of $\law(\sum_{0<s\leq t} \Delta L_s)$ we conclude that $\gamma^0 t \in \supp (\law(L_t))$ for each $t\geq 0$.
\end{proof}

With the aid of Theorem \ref{thm-rocha} and Lemma \ref{lem-drift} we can now obtain the desired generalisation for multivariate quasi-infinitely divisible distributions that are supported in cones.

\begin{theorem}
Let $\mu$ be a quasi-infinitely divisible distribution on $\bR^d$ with characteristic triplet $(A, \nu, \gamma)$ and let $K$ be a cone.
Then the following statements are equivalent:
\begin{enumerate}
\item $\supp (\mu) \subset b + K$ for some $b \in \bR^d$, $\supp (\nu^-) \subset K$ and $\int_{|x|\leq 1} |x| \, \nu^-(\dx) < \infty$.
\item $A=0$, $\supp (\nu^+) \subset K$ and $\int_{|x|\leq 1} |x| \nu^+(\dx) < \infty$.
\end{enumerate}
If the equivalent conditions (i) and (ii) are satisfied, then $\supp (\mu) \subset K$ if and only if the drift of $\mu$ lies in $K$.
\end{theorem}

\begin{proof}
Let $\mu_1$ and $\mu_2$ be infinitely divisible distributions with characteristic triplets and $(A, \nu^+, \gamma)$ and $(0, \nu^-, 0)$ , respectively, so that $\mu_1 = \mu \ast \mu_2$.
To show that (i) implies (ii), suppose that $\supp (\mu) \subset b + K$ for some $b \in \bR^d$, $\supp ( \nu^-)\subset K$ and $\int_{|x|\leq 1} |x| \nu^-(\dx) < \infty$.
Then $\mu \ast \delta_{-b}$ is supported in $K$ and denoting the drift of $\mu_2$ by $b_2 \coloneqq - \int_{|x| \leq 1} x \nu^-(\dx)$, the distribution $\mu_2 \ast \delta_{-b_2}$ is infinitely divisible with characteristic triplet $(0, \nu^-, -b_2)$ and has drift 0.
By Theorem \ref{thm-rocha}, $\mu_2 \ast \delta_{-b_2}$ is supported in $K$ as well, hence also $\mu_2 \ast \mu \ast \delta_{-b-b_2} = \mu_1 \ast \delta_{-b-b_2}$ is supported in $K$.
Using Theorem \ref{thm-rocha} again, since $\mu_1\ast \delta_{-b-b_2}$ is infinitely divisible with characteristic triplet $(A, \nu^+, \gamma-b-b_2)$, it follows that $A= 0$, $\supp(\nu^+) \subset K$ and $\int_{|x| \leq 1} |x| \, \nu^+(\dx) < \infty$.
Moreover, it follows that the drift of $\mu_1 \ast \delta_{-b-b_2}$ is an element of $K$, that is,
\begin{align} \label{support1}
\gamma - b - b_2 - \int_{|x| \leq 1} |x| \nu^+(\dx) \in K.
\end{align}
For the other direction, suppose that (ii) is satisfied and let $b_1 \coloneqq \gamma - \int_{|x| \leq 1} x \, \nu^+(\dx)$ denote the drift of $\mu_1$.
By Theorem \ref{thm-rocha}, the infinitely divisible distribution $\mu_{1} \ast \delta_{-b_1}$ is supported in $K$.
Using \cite[Lem. 24.1]{sato} we obtain $\supp (\mu_2) + \supp (\mu) \subset b_1 + \supp (\mu_1 \ast \delta_{-b_1}) \subset b_1 + K$.
Choosing arbitrary elements $u \in \supp (\mu_2)$ and $v \in \supp (\mu)$, we conclude
\begin{align} \label{support2}
-b_1 + u + \supp (\mu)
\subset -b_1 + \supp (\mu_2) + \supp (\mu_1)
\subset K,
\end{align}
and similarly $\supp (\mu_2 \ast \delta_{-b_1+v}) = -b_1 + v + \supp (\mu_2)  \subset K$.
It follows that $\supp (\mu) \subset b_1-u+K$ and by Theorem \ref{thm-rocha}, since $\mu_2 \ast \delta_{-b_1 + v}$ is infinitely divisible with characteristic triplet $(0, \nu^-, -b_1+v)$, we have that $\supp (\nu^-) \subset K$ and $\int_{|x| \leq 1} |x| \nu^-(\dx) < \infty$.

Now suppose that both (i) and (ii) are satisfied.
With the notations above, note that the drift of $\mu$ is $b_1-b_2$.
If $\mu$ is supported in $K$, then in (i) we can choose $b = 0$ and \eqref{support1} implies that $b_1-b_2 \in K$.
Conversely, if the drift $b_1-b_2$ of $\mu$ is in $K$,  then by Lemma \ref{lem-drift} we can choose  $u = b_2$ and observe from \eqref{support2} that
$\supp (\mu) \subset K + (b_1- b_2) \subset K$, finishing the proof.
\end{proof}

\section{Moments} \label{S8}
\setcounter{equation}{0}

In this section we study the finiteness of the $h$-moment of quasi-infinitely divisible distributions in the case of a submultiplicative function $h$.
Recall that a function $h: \bR^d \to \bR$ is \emph{submultiplicative} if it is non-negative and there exists a constant $C > 0$ such that
\begin{align*}
h(x+y) \leq Ch(x)h(y)
\quad \textrm{for all } x,y \in \bR^d.
\end{align*}
Given an infinitely divisible distribution $\mu$ and a locally bounded submultiplicative and measurable function $h$ on $\bR^d$, $\mu$ has finite $h$-moment if and only if the L\'evy measure $\nu$ of $\mu$ restricted to $\{x \in \bR^d: |x|>1\}$ has finite $h$-moment, i.e. $\int_{\bR^d} h(x) \, \mu(\di x) < \infty$ if and only if $\int_{|x|>1} h(x) \, \nu(\di x) < \infty$, see \cite[Thm. 25.3]{sato}. This does not generalise to quasi-infinitely divisible distributions in the sense that $\int_{\bR^d} h(x) \, \mu(\di x) < \infty$ is equivalent to $\int_{|x|>1} h(x) \, \nu^{\pm}(\di x) < \infty$ (see Example \ref{ex-moments} below), but it generalises in the sense that $\int_{|x|>1} h(x) \, \nu^{+}(\di x) < \infty$ is finite if and only if both $\int_{\bR^d} h(x) \, \mu(\di x)$ and $\int_{|x|>1} h(x) \, \nu^-(\di x)$ are finite. For univariate distributions this was shown in \cite[Thm. 6.2]{lindner}. The proof given there easily generalises to multivariate distributions. We have:

\begin{theorem}
Let $\mu$ be a quasi-infinitely divisible distribution on $\bR^d$ with standard characteristic triplet $(A,\nu,\gamma)$ and let $h: \bR^d \to [0, \infty)$ be a submultiplicative, locally bounded and measurable function.
\begin{enumerate}
\item[\textnormal{(a)}]
Then $(\nu^+)_{| \{x \in \bR^d: |x| > 1\}}$ has finite $h$-moment if and only if both
$\mu$ and $(\nu^-)_{| \{x \in \bR^d: |x| > 1\}}$ have finite $h$-moment, i.e. $\int_{|x|>1} h(x) \, \nu^+(\di x) < \infty$ if and only if $\int_{\bR^d} h(x) \, \mu(\di x) + \int_{|x|>1} h(x) \, \nu^-(\di x) < \infty$.
\item[\textnormal{(b)}]  Let $X$ be a random vector in $\bR^d$ with distribution $\mu$.
Then the following are true:
\begin{enumerate}
    \item[\textnormal{(i)}]  If $\int_{|x|>1} |x| \nu^+(\dx) < \infty$, then the expectation $\bE(X)$ of $X$ exists and is given by
    \begin{align*}
        \bE(X)
        = \gamma + \int_{|x|> 1} x \nu(\dx) = \gamma_m,
    \end{align*}
    which is  the center of $\mu$ as defined in Remark \ref{mean}.
    \item[\textnormal{(ii)}]  If $\int_{|x|>1} |x|^2 \nu^+(\dx) < \infty$, then $X$ has finite second moment and the covariance matrix $\textrm{\rm Cov}(X) \in \bR^{d\times d}$ of $X$ is given by
    \begin{align*}
        \textrm{\rm Cov}(X)
        = A + \int_{\bR^d} xx^{T} \nu(\dx).
    \end{align*}
    \item[\textnormal{(iii)}]  If $\int_{|x|>1}e^{\langle \alpha, x\rangle}\nu^+(\dx) < \infty$ for some $\alpha \in \bR^d$, then $\bE (\re^{\langle \alpha, X\rangle}) < \infty$ and
    \begin{align*}
        \bE(\re^{\langle \alpha, X \rangle})
        = \exp \left(
            \langle \alpha, \gamma \rangle
            + \frac{1}{2} \langle \alpha, A \alpha \rangle
            + \int_{\bR^d} \left(
                \re^{\langle \alpha, x \rangle}-1
                -\langle \alpha, x \rangle \1_{[0,1]}(|x|)
            \right) \nu(\dx)
        \right).
    \end{align*}
\end{enumerate}
\end{enumerate}
\end{theorem}

\begin{proof} The proof given in \cite[Thm. 6.2]{lindner} carries over word by word to the multivariate setting.
\end{proof}

As shown in \cite[Ex. 6.3]{lindner}, for univariate quasi-infinitely divisible distributions it is not true that finiteness of $\int_\bR h(x) \, \mu(\di x)$ implies finiteness of $\int_{|x|>1} h(x) \, \nu^+(\di x)$. Let us give another but simpler example of this phenomenon and also remark on the multivariate setting:

\begin{example} \label{ex-moments}
(a) Let $p\in (0,1/2)$. By  Theorem \ref{cuppens}, the Bernoulli distribution $b(1,p)$ on $\bR$ is quasi-infinitely divisible with quasi-L\'evy measure
$\nu = \sum_{k=1}^{\infty} \frac{(-1)^{k+1}}{k} \left( \frac{p}{1-p} \right)^{k} \delta_{k}$.
Especially, we obtain  $\nu^+ = \sum_{k=0}^{\infty} \frac{1}{2k+1} \left( \frac{p}{1-p} \right)^{2k+1} \delta_{2k+1}$.
Let $c > -\log \frac{p}{1-p} >0$.
The function $g: \bR \to \bR$, $x \mapsto e^{cx}$ is submultiplicative and by  the monotone convergence theorem it holds
\begin{align*}
    \int_{|x|> 1} \re^{cx} \nu^+(\dx)
    &= \sum_{k=1}^{\infty} \frac{1}{2k+1} \left( \frac{p}{1-p} \right)^{2k+1} \re^{(2k+1)c} = \infty,
\end{align*}
since $\frac{p}{1-p} \re^{c} > 1$.
Similarly, $\int_{|x|> 1} \re^{cx} \nu^-(\dx) = \infty$, although the Bernoulli-distribution $b(1,p)$ has finite $g$-moment.\\
(b) Let $p\in (0,1/2)$ and $X_1,\ldots, X_d$ be independent $b(1,p)$-distributed and define $X=(X_1,\ldots, X_d)^T$. By Proposition \ref{prop-indep}, $\law(X)$ is a quasi-infinitely divisible distribution on $\bR^d$ whose quasi-L\'evy measure $\widetilde{\nu}$ is concentrated on the axes. Let $c>-\log \frac{p}{1-p}$ and consider the function $h:\bR^d \to [0,\infty)$ given by $h(x) = \re^{\langle c e_1,x\rangle}$, where $e_1$ is the first unit vector in $\bR^d$. Then $h$ is submultiplicative and $\int_{|x|>1} h(x) \, \widetilde{\nu}^{\pm} (\di x) = +\infty$ by (a), although $\law(X)$ has bounded support and hence finite  $h$-moment.
\end{example}



\begin{thebibliography}{99}


\bibitem{Applebaum2009} Applebaum, D. (2009). \emph{L\'evy Processes and Stochastic Calculus}. 2nd edition. Cambridge Univ. Press, Cambridge.


\bibitem{Armstrong} Armstrong, M.A. (1983). \emph{Basic Topology}. Springer, New York.

\bibitem{AN13}
Aoyama, T. and Nakamura, T. (2013).
{Behaviors of multivariable finite Euler products in probabilistic view}.
\emph{Math. Nachr.} {\bf 286}, 1691--1700.

\bibitem{Bertoin1996} Bertoin, J. (1996). \emph{L\'evy Processes}.  Cambridge Univ. Press, Cambridge.

\bibitem{BertoinDoneyMaller2008}{Bertoin, J., Doney, R.A. and Maller, R.A. (2008). Passage of L\'evy processes across power law boundaries at small times. \emph{Ann. Probab.} {\bf 36}, 160--197.}

\bibitem{berger}{Berger, D. (2018). On quasi-infinitely divisible distributions with a point mass. {\it Math.
Nachr.} {\bf 292}, 1674--1684.}

\bibitem{cramerwold}{Berger, D. and Lindner, A. (2020). A Cram\'er--Wold device for infinite divisibility of $\mathbb{Z}^d$-valued distributions. Preprint. Available at \emph{arXiv:2011.08530}.}


\bibitem{ChhaibaDemniMouayn2016} Chhaiba, H., Demni, N. and Mouayn, Z. (2016).
    Analysis of generalized negative binomial distributions attached to hyperbolic Landau levels.
    \emph{J. Math. Phys.} {\bf 57} (7), 072103, 14pp.


\bibitem{Cuppens1969} Cuppens, R. (1969). Quelques nouveaux r\'esultats en arithm\'etique des lois de probabilit\'e. In \emph{C.R. Colloqu. C.N.R.S., Les probabilit\'es sur les structures alg\'ebraiques}, pp. 97--112, C.N.R.S. Paris.



\bibitem{cuppens75}
    Cuppens, R. (1975). \emph{Decomposition of Multivariate Probabilites}. Academic Press, New York.



\bibitem{DemniMouayn2015} Demni, N. and Mouayn, Z. (2015). Analysis of generalized Poisson distributions associated with higher Landau leves. \emph{Infin. Dimens. Anal. Quantum Probab. Relat. Top.} {\bf 18} (4), 1550028, 13pp.

\bibitem{Doney2004} Doney, R. (2004). Small-time behaviour of L\'evy processes. \emph{Electron. J. Probab.} {\bf 9} (8), 209--229.

\bibitem{Doney2007} Doney, R.A. (2007). \emph{Fluctuation Theory for L\'evy Processes.} Lecture Notes in Mathe\-matics, {\bf 1897}. Springer, Berlin.

\bibitem{DoneyMaller2002} Doney, R.A. and Maller, R.A. (2002). Stability and attraction to normality for L\'evy processes at zero and at infinitey. \emph{J. Theoret. Probab.}
    {\bf 15} (3), 751--792.

\bibitem{GK68}
Gnedenko, B.V. and Kolmogorov, A.N. (1968). \emph{Limit Distributions for Sums of Independent Random Variables}. Rev. edition. Addison Wesley, Reading. Translation of the Russian original of 1949.


\bibitem{KadankovaSimonWang2020} Kadankova, T., Simon, T. and Wang, M. (2020). On some new moments of gamma type. \emph{Statist. Probab. Lett.} {\bf 165}, 108854, 7 pp.


\bibitem{kallenberg}{Kallenberg, O. (1981). Splitting at backward times in regenerative sets. {\it Ann. Prob.} {\bf 9} (5), 781--799.}

\bibitem{Khartov2019} Khartov, A. (2019). Compactness criteria for quasi--infinitely divisible distributions on the integers. \emph{Stat. Probab. Lett.} {\bf 153}, 1--6.

\bibitem{Kyprianou2014} Kyprianou, A.E. (2014). \emph{Fluctuations of L\'evy Processes with Applications. Introductory Lectures}. 2nd edition. Springer, Heidelberg.

\bibitem{lindner}{Lindner, A., Pan, L. and Sato, K. (2018). On quasi-infinitely divisible distributions. {\it Trans. Amer. Math. Soc. }{\bf 370}, 8483--8520.}

\bibitem{LiSa2011} Lindner, A. and Sato, K. (2011).
Properties of stationary distributions of a sequence of generalised Ornstein-Uhlenbeck processes. {\it Math. Nachr.} {\bf 284}, 2225--2248.

\bibitem{Linnik64}
Linnik, Yu.V. (1964). \emph{Decomposition of Probability Distributions.} Oliver and Boyd Ltd, Edingburgh.

\bibitem{LO77}
Linnik, Yu.V. and Ostrovski\u{i}, I.V. (1977). \emph{Decomposition of Random Variables and Vectors}. American Mathematical Society, Providence, Rhode Island.

\bibitem{Nakamura13}
Nakamura, T. (2013). {A quasi-infinitely divisible characteristic function and its exponentation}. \emph{Stat. Probab. Lett.} {\bf 83}, 2256--2259.


\bibitem{Nakamura15}
Nakamura, T. (2015). {A complete Riemann zeta distribution and the Riemann hypothesis}.
\emph{Bernoulli} {\bf 21}, 604--617.

\bibitem{Passeggeri2020b}{Passeggeri, R. (2020). A density property for stochastic processes. Preprint. Available at \emph{arXiv:2010.07752}.}


\bibitem{Passeggeri2020a} Passeggeri, R. (2020). Spectral representations of quasi-infinitely divisible processes. \emph{Stoch. Process. Appl.} {\bf 130}, 1735--1791.

\bibitem{Pears}Pears, A. R. (1975). {\it Dimension Theory of General Spaces}. Cambridge University Press. Cambridge. \label{Pears}

\bibitem{Rocha-Sato} Rocha-Arteaga, A. and Sato, K. (2019).  \emph{Topics in Infinitely Divisible Distributions and L\'evy Processes. Revised Edition.} Springer Briefs in Probability and Mathematical Statistics, Springer, Cham.

\bibitem{SamorodnitskyTaqqu} Samorodnitsky, G. and Taqqu, M.S. (1994). \emph{Stable Non-Gaussian Random Processes}. Boca Raton: Chapman \& Hall, New York.

\bibitem{sato}{Sato, K. (2013). {\it L\'evy Processes and Infinitely Divisible
  Distributions}. Corrected Printing with Supplement. Cambridge University Press, Cambridge.}

\bibitem{skorohod} Skorohod, A.V. (1991). \emph{Random Processes with Independent Increments.} Translated from the second Russian edition. Kluwer Academic Publishers, Dordrecht.

\bibitem{ZhangLiuLi} Zhang, H., Liu, Y. and Li, B. (2014). Notes on discrete compound Poisson model with
applications to risk theory. \emph{Insurance Math. Econom.} {\bf 59}, 325--336.


\end{thebibliography}
\end{document}